\documentclass[a4paper,11pt,leqno]{amsart}
\usepackage{amsmath}
\usepackage{amsthm}
\usepackage{amsfonts}
\usepackage{amssymb}

\usepackage{mathrsfs}

\usepackage{enumitem}

\usepackage[abs]{overpic}

\usepackage[usenames]{color}
\usepackage[all]{xy}

\usepackage{url}
\usepackage[all]{xy}
\usepackage{graphicx}
\usepackage{latexsym}
\usepackage[cp850]{inputenc}
\usepackage{epsfig}
\usepackage{psfrag}
\usepackage{amsfonts}
\newtheorem{theorem}{Theorem}[section]{\bf}{\it}
\newtheorem{lemma}[theorem]{Lemma}{\bf}{\it}
{\bf}{\it}
{\bf}{\it}
{\bf}{\it} % Main Theorems, numbered A,B,...
{\bf}{\it}
\newtheorem*{theorem*}{Theorem}

\newtheorem*{namedtheorem}{\theoremname}
\newcommand{\theoremname}{testing}

{\bf}{\it}
{\bf}{\it}

\theoremstyle{remark}
\newtheorem*{remark}{Remark}%[theorem]{Remark}{\bf}{\it}
\newtheorem{definition}[theorem]{Definition}

\theoremstyle{definition}
\theoremstyle{remark}

\numberwithin{equation}{section}

\newcommand{\R}{\mathbb R}
\newcommand{\Z}{\mathbb Z}

\newcommand{\N}{\mathbb N}

\newcommand{\dist}{{\operatorname{dist}\,}}
\newcommand{\id}{{\operatorname{id}}}

\newdimen\vintkern\vintkern11pt
\def\vint{-\kern-\vintkern\int}

\newcommand{\dx}{\;\mathrm{d}x}

\newcommand{\bS}{\mathbb{S}}

\newcommand{\vol}{\mathrm{vol}}

\newcommand{\interior}{\mathrm{int}\;}

\newcommand{\Lip}{\mathrm{Lip}}

\newcommand{\cB}{\mathcal{B}}

\newcommand{\cD}{\mathcal{D}}
\newcommand{\cI}{\mathcal{I}}
\newcommand{\cN}{\mathcal{N}}

\newcommand{\sA}{\mathsf{A}}
\newcommand{\sB}{\mathsf{B}}
\newcommand{\sL}{\mathsf{L}}
\newcommand{\sM}{\mathsf{M}}
\newcommand{\sT}{\mathsf{T}}
\newcommand{\sC}{\mathsf{C}}

\newcommand{\Corner}{\mathrm{Corner}}

\newcommand{\Terminal}{\mathrm{Terminal}}

\newcommand{\inner}{\mathrm{inner}}

\renewcommand{\emptyset}{\varnothing}

\begin{document}

\title[Sobolev regular wild involution of the $3$-sphere]{Bing meets Sobolev}
%\date{\today}

\author{Jani Onninen}
\address{Department of Mathematics, Syracuse University, Syracuse,
NY 13244, USA \and  Department of Mathematics and Statistics, P.O.Box 35 (MaD) FI-40014 University of Jyv\"askyl\"a, Finland}
\email{jkonnine@syr.edu}

\author{Pekka Pankka}
\address{Department of Mathematics and Statistics, P.O. Box 68 (Gustaf H\"allstr\"omin katu 2b), FI-00014 University of Helsinki, Finland}
\email{pekka.pankka@helsinki.fi}

\thanks{This work was supported in part by the Academy of Finland project \#297258 and the NSF grant  DMS-1700274.}

\date{\today}

\dedicatory{Dedicated to Pekka Koskela on the occasion of his 59th birthday.}

\begin{abstract}
We show that, for each $1\le p < 2$, there exists a wild involution $\bS^3\to \bS^3$ in the Sobolev class $W^{1,p}(\bS^3,\bS^3)$. 
\end{abstract}

\maketitle

\setcounter{tocdepth}{3}
%\tableofcontents

\section{Introduction}

A special case of a classical result of P.A.\;Smith \cite{Smith} states that the fixed point set of an involution $\bS^3\to \bS^3$ is homeomorphic to either $\bS^1$ or $\bS^2$. More precisely, the fixed point set is homeomorphic to $\bS^2$ for orientation-reversing involutions and homeomorphic to $\bS^1$ for orientation-preserving (non-identity) involutions. Recall that a homeomorphism $f\colon \bS^3\to \bS^3$ is an \emph{involution} if $f\circ f = \id$,  that is, $f^{-1} = f$. We call orientation-reversing involutions \emph{reflections}.

%In the case of 
For a $C^1$-involution, the fixed point set is a smooth submanifold by a similarly classical result of Bochner \cite{Bochner}. Topological involutions, on the other hand, exhibit much wilder behavior. In a celebrated paper \cite{Bing} Bing constructed an example of an orientation-reversing wild involution of $\bS^3$ which has a wildly embedded $2$-sphere as its fixed point set. Montgomery and Zippin \cite{Montgomery-Zippin} modified Bing's construction to obtain an orientation-preserving wild involution of $\bS^3$ having a wildly embedded circle as its fixed point set. 
Recall that an embedded $k$-sphere $S\subset \bS^3$ is \emph{tamely embedded} (or \emph{tame} for short) if there is a homeomorphism $h \colon \bS^3 \to \bS^3$ for which $h(S) = \bS^k \subset \bS^3$, and \emph{wildly embedded} otherwise.

In what follows, we say than an involution $\bS^3\to \bS^3$ is \emph{tame} (resp.\;\emph{wild}) if its fixed point set is tame (resp.\;wild). In this terminology $C^1$-involutions of $\bS^3$ are tame. Whereas this result in the orientation-reversing case is a straight forward consequence of the collarability of smooth $2$-spheres in $\bS^3$ and the generalized Schoenflies theorem~\cite{BrownSc}, the case of orientation-preserving involutions is a part of the Smith conjecture, proved first by Waldhausen~\cite{Waldhausen} 
%proved the Smith conjecture for 
in the special case of diffeomorphisms of order $2$ and Morgan and Bass \cite{MB1984} in the general case. Recall that in two dimensions involutions of $\bS^2$ are tame; see Brouwer \cite{Brouwer}.

{It is clear from the setting that Bochner's result on the smoothness of the fixed point set does not extend below $C^1$-smoothness.}
%It is clear from the setting that, in terms of smoothness of the fixed point set, Bochner's result does not extend below $C^1$-smoothness. 
Indeed, it suffices to conjugate $\lambda \colon (x_1,x_2,x_3) \mapsto (-x_1,x_2,x_3)$ by a non-smooth bilipschitz homeomorphism $h \colon \R^3 \to \R^3$, $(x_1,x_2,x_3) \mapsto (x_1+|x_2|, x_2, x_3)$, to obtain an involution $f = h \circ \lambda \circ h^{-1}$ of $\R^3$ for which the fixed point set $h(\{0\} \times \R^2)$ is not a smooth submanifold of $\R^3$. Conjugation of $f$ by the stereographic projection and extending the obtain homeomorphism to $e_4\in \bS^3$ yields now an example of desired type.

{The question, whether the fixed point set is wildly embedded under weaker assumptions on smoothness, has no such easy answer.}
%However, the question, whether the fixed point set is wildly embedded under weaker assumptions on smoothness, has no such easy answer. 
In \cite{Heinonen-Semmes} Heinonen and Semmes asked (Question 26) whether there exists a wild quasiconformal reflection of $\bS^3$; the answer to this question is open even in the case of bilipschitz involutions of $\bS^3$.\footnote{It is announced in \cite{Hamilton} that quasiconformal reflections are tame.} 

In this article we consider the Sobolev regularity of wild involutions related to  Bing's construction~\cite{Bing}. Bing's construction has been used to obtain several others wild constructions in quasiconformal geometry, most notably by Freedman and Skora \cite{Freedman-Skora} for wild quasiconformal actions and by Semmes \cite{Semmes} for quasisymmetric non-parametrization theorems; see also Heinonen-Wu \cite{Heinonen-Wu} and \cite{PV,PW} for similar results in dimensions $n>3$. However, to our knowledge the regularity of wild involutions have not been considered in the literature apart from the modulus of continuity; see Bing \cite{Bing1988}.   Our main result reads as follows.

\begin{theorem}
\label{thm:main}
For $p\in [1,2)$ %$1\le p < 2$ 
there exist an orientation-preserving and an orientation-reserving wild involution of $\bS^3$ in the Sobolev space $W^{1,p}(\bS^3,\bS^3)$.
\end{theorem}

{As mentioned,} 
our constructions are Sobolev space versions of the constructions of Bing and Montgomery--Zippin. The wild involution $f \colon \bS^3 \to \bS^3$ is a quotient of a linear involution $\iota \colon \bS^3\to \bS^3$  
\[
\xymatrix{\bS^3 \ar[r]^\iota \ar[d]_\phi & \bS^3 \ar[d]^\phi  \\ \bS^3 \ar[r]^f & \bS^3 }
\]
in a monotone map $\phi \colon \bS^3 \to \bS^3$ associated to the \emph{Bing's double} in \cite{Bing}. In the orientation-reversing case, the isometric involution $\iota$ is $(x_1,x_2,x_3,x_4) \mapsto (-x_1,x_2,x_3,x_4)$ and the fixed point set is a wild $2$-sphere. In the orientation-preserving case $\iota$ is the involution $(x_1,x_2,x_3,x_4) \mapsto (-x_1,-x_2, x_3,x_4)$ and the fixed point set is a wild $1$-sphere.

%In the orientation-reversing case, which produces a wild $2$-sphere as the fixed point set, the isometric involution $\iota$ is $(x_1,x_2,x_3,x_4) \mapsto (-x_1,x_2,x_3,x_4)$, and in the orientation preserving case $\iota$ is the involution $(x_1,x_2,x_3,x_4) \mapsto (-x_1,-x_2, x_3,x_4)$.

To obtain a Sobolev regular wild involution, we consider a modified version of the defining sequence for Bing's double. We show, in spirit of \cite{DP}, that there is a defining sequence, consisting of solid $3$-tori, which are uniformly bilipschitz to certain model cubical $3$-tori in their inner metric. This allows us to compute the derivative of $f$ from the twist maps in Bing's shrinking process \cite{Bing1988} in the complement of  the wild Cantor set obtained as the image of Bing's double in the map $\phi$.
%a wild Cantor set associated to Bing's construction. 
For given $p\in [1,2)$, by choosing the tori with sufficiently small volumes at each stage of the defining sequence, we obtain the $L^p$-integrability of $Df$. 

Heuristically, the exponent $p=2$ is associated to the balance between volume of the solid tori and the local bilipschitz constant of twist maps rotating the interiors of these tori. %, which are given by Bing's shrinking process \cite{Bing1988}. 
We do not know whether we may reach the $W^{1,2}$-Sobolev regularity for $f$ using our method. However, we obtain {the integrability of} the adjoint $D^\# f$ of $Df$. %is integrable. 
{Namely,} for the wild involution $f$ it follows from the change of variables formula that
\[
\int_{\bS^3} |Df(x)|^p\;\mathrm{d}x = \int_{\bS^3} \frac{|D^\# f(x)|^p}{J_f(x)^{p-1}}\;\mathrm{d}y
\]
for $p\in [1,2)$, where $J_f$ is the Jacobian determinant of the differential $Df$.  %We find this interesting in particular in the context of the question of Heinonen and Semmes on quasiconformal wild involutions. 
In  spirit of the question of Heinonen and Semmes, it would be interesting to know whether there exist wild involutions in $W^{1,p}(\bS^3,\bS^3)$ for $p\ge 2$.

During the course of the proof of Theorem \ref{thm:main}, we obtain that the monotone map $\phi \colon \bS^3 \to \bS^3$ is in the same Sobolev space $W^{1,p}$ as the involution $f\colon \bS^3\to \bS^3$. It should, however, be noted that  the Sobolev regularity of $f$ does not follow immediately from the Sobolev regularity of $\phi$  since \emph{a priori} the composition $\phi\circ \iota \circ \phi^{-1}$, in the complement of the singular set, is not in $W^{1,1}$. An additional cancellation property of $f$ is needed to prove the seeked regularity.

%We find it interesting that the direct approach to calculate the Sobolev regularity, using the fact that $f$ is a conjugation of $\iota$ by $\phi$, does not lead to any Sobolev regularity of $f$. Indeed, the methods yields that both $f$ and $\phi$ belong to the same space $W^{1,p}(\bS^n,\bS^n)$ and the regularity of $f$ follows from a local cancellation property. 

\subsection{Connection to nonlinear elasticity}

Our original interest to investigate the Sobolev regularity of homeomorphisms comes from the theory of nonlinear elasticity -- in particularly the Ball-Evans approximation  problem~\cite{Ba}.  It asks if a $W^{1,p}$-Sobolev homeomorphism can be approximated in the strong topology of $W^{1,p}$ by piecewise affine invertible mappings. J. M. Ball attributes this question to L. C. Evans and points out its relevance to the  regularity of minimizers of neohookean energy functionals.  In the context of nonlinear elasticity~\cite{Anb, Bac, Cib}, one typically deals with two or three dimensional models. The Ball-Evans problem is completely understood in the planar case~\cite{HP, IKOapprox} and wildly open in dimension three. %The motivation to study topological invariants of Sobolev homeomorphisms comes from these questions in three dimensions models.  

%It is worth noting that a breakthrough result in higher dimension ($n \ge 4$) is the result of Campbell, Hencl, and Tengvall~\cite{CHT}, which generalizes an earlier result by Hencl and Vejnar~\cite{HV}.

It is worth noting that an important result in the Ball--Evans problem in dimension $n=4$ is due to Hencl and Vejnar~\cite{HV} for $W^{1,1}$-homeomorphisms and in dimensions $n \ge 4$ by Campbell, Hencl, and Tengvall~\cite{CHT}: \emph{for each $1\le p < [n/2]$, there is a $W^{1,p}$-Sobolev homeomorphism, which cannot be approximated by piecewise affine homeomorphisms}.
%It states that there is a $W^{1,p}$-Sobolev homeomorphism, $1 \le p<[n/2]$, which cannot be approximated by piecewise affine invertible mappings. 
Here $[a]$ denotes the integer part of $a$.

The connection between wild involutions and the Ball--Evans problem is via cellular mappings. We finish this introduction by discussing this connection between geometric topology and nonlinear elasticity.

To build a viable theory of minimization problems for three or higher dimensional models, we come to the question on enlargement of the class Sobolev homeomorphisms. Clearly, enlarging the set of the admissible mappings may change the nature of the energy-minimal solutions. In two dimensions, the classical Youngs approximation theorem~\cite{Youngs} states that a  continuous map between $2$-spheres is monotone if and only if it is a uniform limit of homeomorphisms. Using a Sobolev variant of Youngs approximation theorem~\cite{IOmono}, we may enlarge the minimization to monotone mappings and to avoid the Lavrentiev phenomenon. 

%one applies in two dimensional models a Sobolev variant of the classical Youngs approximation theorem~\cite{IOmono} which  asserts that   a  continuous map between $2$-spheres  is monotone  if and only if it is a uniform limit of homeomorphisms. 
 
In three dimensions it is possible to construct monotone mappings of the $3$-sphere onto itself which cannot be uniformly approximated by homeomorphisms~\cite{BingDeomposition}. However, the approximation  is possible for {\it cellular} mappings; mappings whose inverse image of a point is an intersection of a decreasing sequence of
$n$-cells, the notion introduced by Brown~\cite{Br}. The monotone map $\phi$ obtained in the construction of the Bing's double is an example of a cellular mapping. 

Armentrout showed~\cite{Ar} that cellular mappings of an $3$-manifold onto itself can be approximated by homeomorphisms; see also Siebenmann~\cite{Si}. We refer to a book of Daverman \cite{Daverman-book} for the development of these mappings as a part of the theory of decomposition spaces and manifold recognition problems.

A Sobolev variant of the result of Armentrout would convince us that the Lavrentiev phenomenon in three dimensional minimization problems can be avoided by adopting  Sobolev cellular mappings. Nevertheless the theory of Sobolev cellular mappings is still in its infancy. 

%This article is organized as follows. In Section 2. we discuss the general properties of bilipschitz tori.

\bigskip
\noindent 
{\bf Acknowledgements} Authors thank Piotr Haj\l asz~\cite{sobo_meets_poin} and Pekka Koskela~\cite{sobo_meets_poin} for inspiration {regarding the title.} %for inspiration in naming this article.

%\section{Sobolev preliminaries}

\section{Cubical preliminaries}

Let $\cD$ be the collection of dyadic cubes in $\R^3$ of side length at most $1$; that is, cubes $Q = 2^{-k}(v + [0,1]^3)$, where $v\in \Z^3$ and $k\in \Z_+$. Given a subcollection $\sC\subset \cD$ , we denote their union $|\sf C| = \bigcup \sf C$. We call the number of cubes $\#\sC$ in $\sC$ the \emph{cubical length of $\sC$}.

We say that cubes $Q$ and $Q'$ in $\cD$ are \emph{adjacent} (denoted $Q \sim Q'$) if $Q\cap Q'$ is a common face of both cubes. In particular, adjacent cubes have the same side length. We also say that cubes $Q$ and $Q'$ \emph{meet} if $Q\cap Q'\ne \emptyset$.

\begin{figure}[h!]
\begin{overpic}[scale=0.2,unit=1mm]{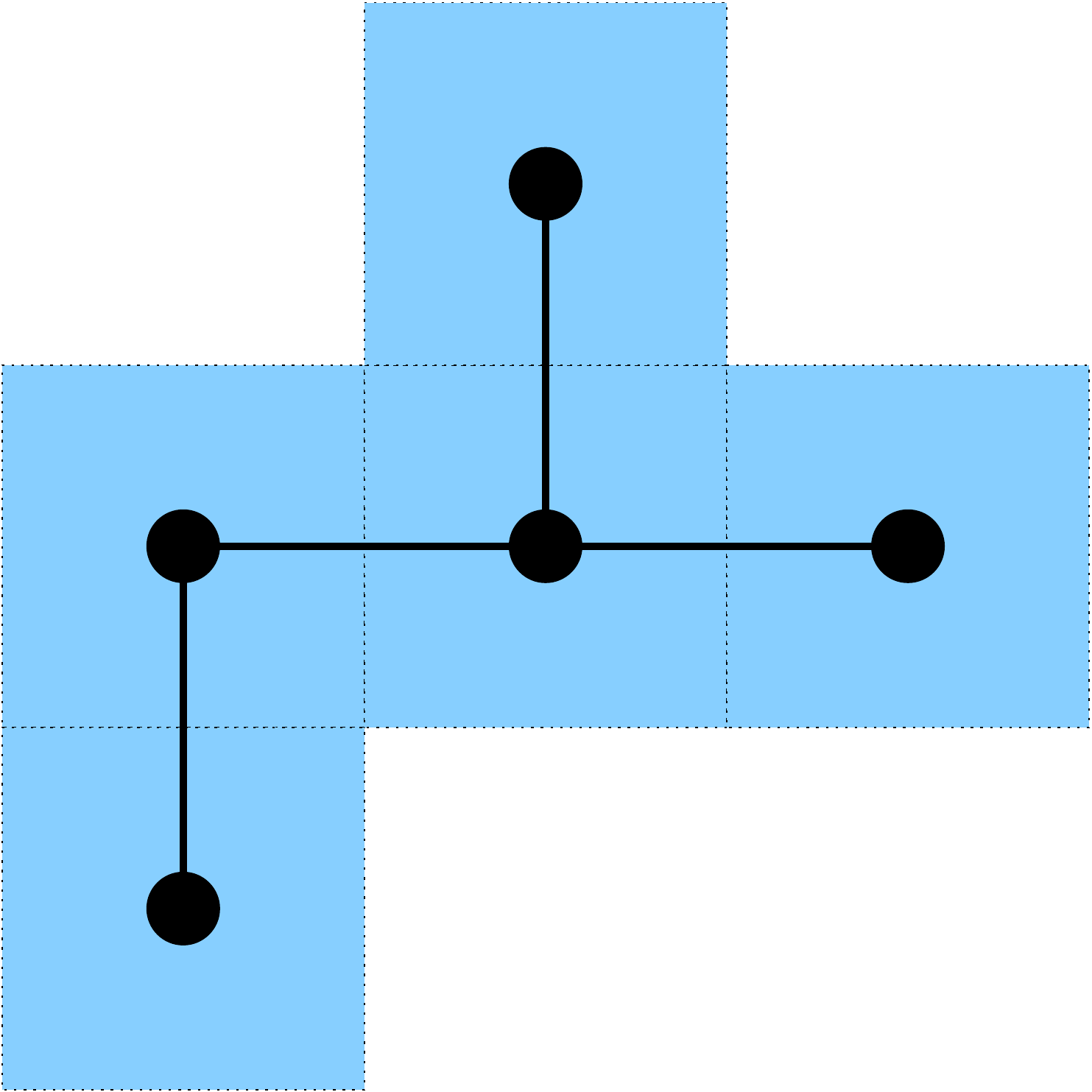} % add: grid
%\put(13,13){\tiny $D$}
%\put(45,13){\tiny $D'$}
\end{overpic}
\caption{A collection $\sf C$ of five adjacent cubes and their adjacency graph $\Gamma(\sf C)$.}
\label{fig:Adjacent_cubes}
\end{figure}

Given a collection $\sC$ of cubes $\cD$ of the same side length, the adjacency graph $\Gamma(\sC)$ is the graph having cubes $Q$ in $\sC$ as vertices and pairs $\{Q,Q'\}$ of adjacent cubes $Q$ and $Q'$ in $\sC$ as edges; see Figure \ref{fig:Adjacent_cubes} for an example of an adjacency graph. For each $Q\in \sC$, we denote $\cN_\sC(Q)$ the collection of all cubes in $\sC$ which intersect $Q$. We call $\cN_\sC(Q)$ the cubical neighborhood of $Q$ in $\sC$. Note that $\cN_\sC(Q)$ contains all cubes adjacent to $Q$, but may contain also other cubes. We also denote $\sigma(\sC)$ the common side length of the cubes in $\sC$.

%An adjacency graph $\Gamma(\sf C)$ is an \emph{arc} if $\Gamma(\sC)$ is a finite tree having valency at most $2$ at each vertex.

\subsection{Cubical loops and arcs}
In what follows, we consider mainly cubical arcs and loops, which are collactions of cubes $\sC$ for which the graph $\Gamma(\sC)$ has valence at most two. For the definitions of cubical loops and arcs, we distinguish first three special classes of cubes. 
\begin{definition}
A cube $Q\in \sC$ is an \emph{$I$-cube} if $Q$ is adjacent to exactly two cubes $Q_+$ and $Q_-$ in $\sC$ and $Q_+ \cap Q_-=\emptyset$.
\end{definition}
Here the heuristic idea is that the union $Q_+ \cup Q \cup Q_-$ is an image of $[0,1]^2\times [0,3]$ under a similarity map.

\begin{definition}
A cube $Q\in \sC$ is a \emph{corner of $\sC$} if $Q$ is (again) adjacent to exactly two cubes $Q_+$ and $Q_-$ in $\sC$ and $Q_+ \cap Q_-\ne \emptyset$. 
\end{definition}
Now the heuristic idea is that the union $Q_+\cup Q \cup Q_-$ is an image of the union of $(e_1+[0,1]^3) \cup [0,1]^3 \cup (e_2+[0,1]^3)$ under a similarity map. For illustrations of an $I$-cube and a corner, see Figure \ref{fig:Corner22}. We denote $\cI(\sC)$ and $\Corner(\sC)$ the collections of $I$-cubes and corners of $\sC$, respectively.

Finally, we define terminal cubes of a collection.
\begin{definition}
A cube $Q\in \sC$ is a \emph{terminal cube of $\sC$} if $Q$ is a leaf in $\Gamma(\sC)$, that is, $Q$ is adjacent to exactly one cube in $\sC$.
\end{definition}

\begin{figure}[h!]
\begin{overpic}[scale=0.2,unit=1mm]{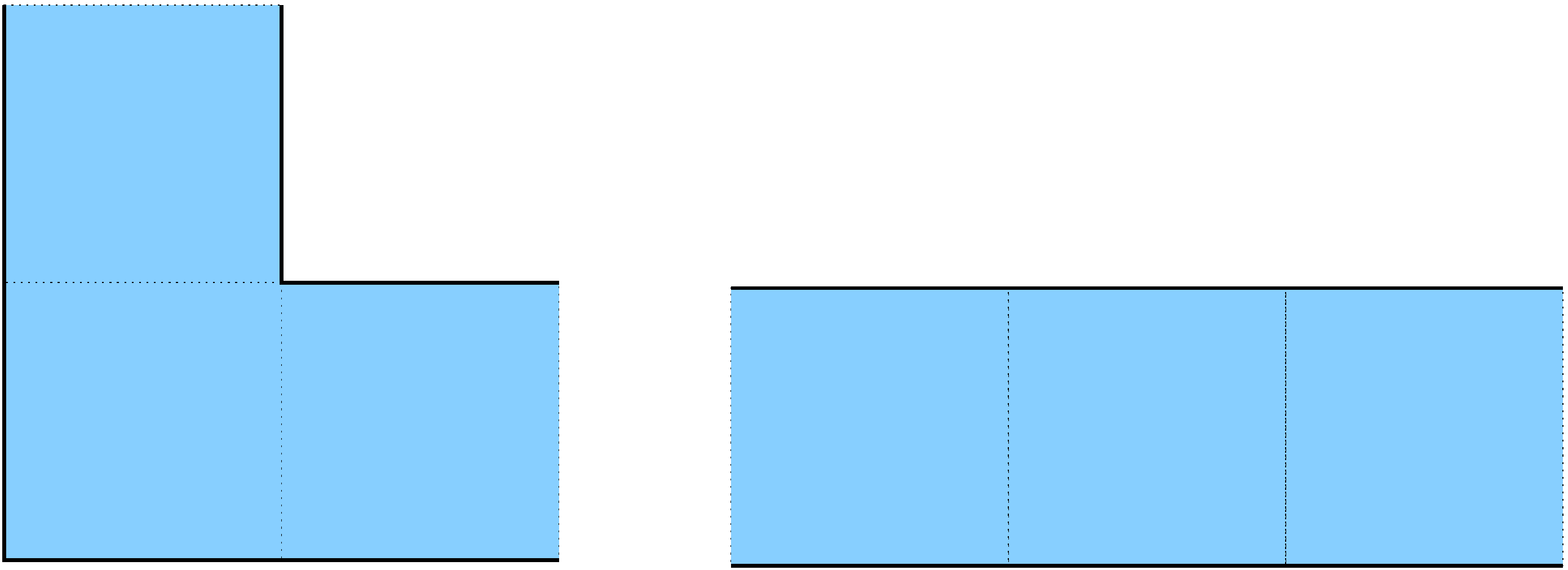} % add: grid
%\put(13,13){\tiny $D$}
%\put(45,13){\tiny $D'$}
\end{overpic}
\caption{A corner and a segment.}
\label{fig:Corner22}
\end{figure}

Having these three definitions at our disposal, we may define cubical loops and cubical arcs.
\begin{definition}
A finite collection $\sL$ is a \emph{cubical loop} if the cubes in $\sL$ have the same side length, each cube in $\sL$ is either an $I$-cube or a corner, and cubical neighborhoods of corners of $\sL$ are mutually disjoint. A loop $\sL$ is a \emph{model loop} if it has exactly four corners. 
\end{definition}

Note that, by finiteness of $\sL$, the adjacency graph $\Gamma(\sL)$ of a cubical loop $\sL$ is always a cycle; see Figure \ref{fig:CubicalLoop} for an example. In particular, $|\sL|$ is homeomorpic to the solid $3$-torus $\bar B^2\times \bS^1$. 

\begin{figure}[h!]
\begin{overpic}[scale=0.10,unit=1mm]{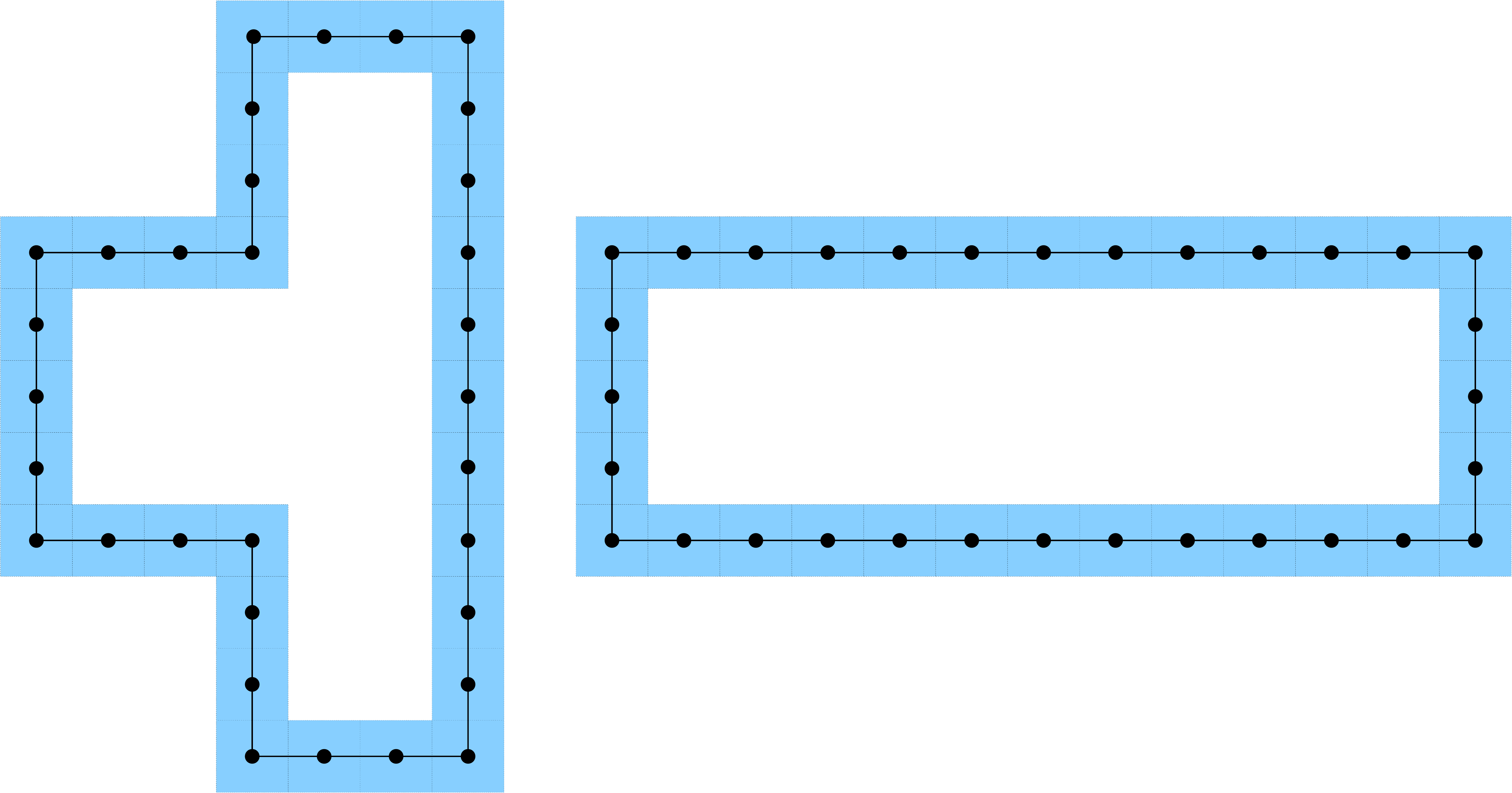} % add: grid
%\put(13,13){\tiny $D$}
%\put(45,13){\tiny $D'$}
\end{overpic}
\caption{A cubical loop and a model loop of the same cubical length.}
\label{fig:CubicalLoop}
\end{figure}

\begin{definition}
A finite collection $\sA$ is a \emph{cubical arc} if the cubes in $\sA$ have the same side length, $\sA$ has two terminal cubes, each other cube in $\sL$ is either an $I$-cube or a corner, and cubical neighborhoods of corners of $\sL$ are mutually disjoint. A cubical arc is a \emph{segment} if it has no corners. 
\end{definition}

\begin{remark}
We would like to emphasize a technical point. In what follows, we consider cubical loops and arcs primarily as combinatorial objects. In particular, we keep distinguishing a cubical loop $\sL$ from its union $|\sL|$. Formally an arc $\sL$ is a collection of cubes in $\R^3$, whereas the union $|\sL|=\bigcup_{Q\in {\sL}} Q$ is a subset in $\R^3$. 
\end{remark}

%\subsection{Cubical arcs}

%A finite collection $\sf A\subset \cD$ is a \emph{cubical arc (in $\R^3$)} if $\Gamma(\sf A)$ is an arc, that is and, whenever cubes $Q$ and $Q'$ in $\sf A$ {\color{blue} have non-empty intersection}, there exists a cube $Q''$ in $\sf A$ {\color{blue}for which} $Q \sim Q''$ and $Q'' \sim Q'$; we may take $Q''=Q'$ if $Q\sim Q'$.  A collection $\sf A'\subset \sf A$ is a \emph{subarc of $\sf A$} if $\sf A'$ is an arc.  We denote $s(\sf A)$ the common side length of cubes in $\sA$.

\begin{remark}
We note also that cubical arcs admit an alternative characterization. A finite collection $\sf A$ is a cubical arc if and only if $|\sA|$ is an $3$-cell and there exists a linear order $Q_1,\ldots, Q_k$ of cubes in $\sA$, where $k=\# \sA$ is the cubical length of $\sA$, so that $Q_j \sim Q_{j+1}$ for each $j\in \{1,\ldots, k-1\}$. 
\end{remark}

%Let $\sf A$ be a cubical arc. A cube $Q$ is a \emph{corner in $\sf A$} if $Q$ is adjacent to two other cubes $Q'$ and $Q''$ in $\sf A$ so that $Q'\cap Q''\ne \emptyset$. We denote $\Corner(\sf A)$ the collection of all corner cubes of $\sf A$. We say a cubical arc is a \emph{segment} if it has no corners. See Figure \ref{fig:Corner22} for examples.

The fact, which will play a crucial role in the forthcoming discussion, is that -- in their inner geometry -- a cubical arc is locally uniformly bilipschitz equivalent to a segment. We formulate this precisely in Lemma \ref{lemma:inner_bilip} after introducing some terminology for the statement.

The collection of $I$-cubes in $\sf A$ is naturally {partitioned} into maximal subarcs of $\sf A$, which we call \emph{$I$-blocks}. Given an $I$-block $\sf B$ in $\sf A$, we call $\sf B\setminus \Terminal(\sf B)$ a \emph{reduced $I$-block}. Note that the terminal cubes of an $I$-block are adjacent to either corners or terminal cubes of $\sf A$. The terminal cubes of a reduced $I$-block are adjacent to cubes in cubical neighborhood of corners or terminal cubes of $\sf A$. We denote $\cB(\sf A)$ and $\cB^\flat(\sf A)$ the collections of $I$-blocks and reduced $I$-blocks in $\sf A$, respectively. See Figure \ref{fig:CubeTypes} for an illustration of such partition.

\begin{figure}[h!]
\begin{overpic}[scale=0.20,unit=1mm]{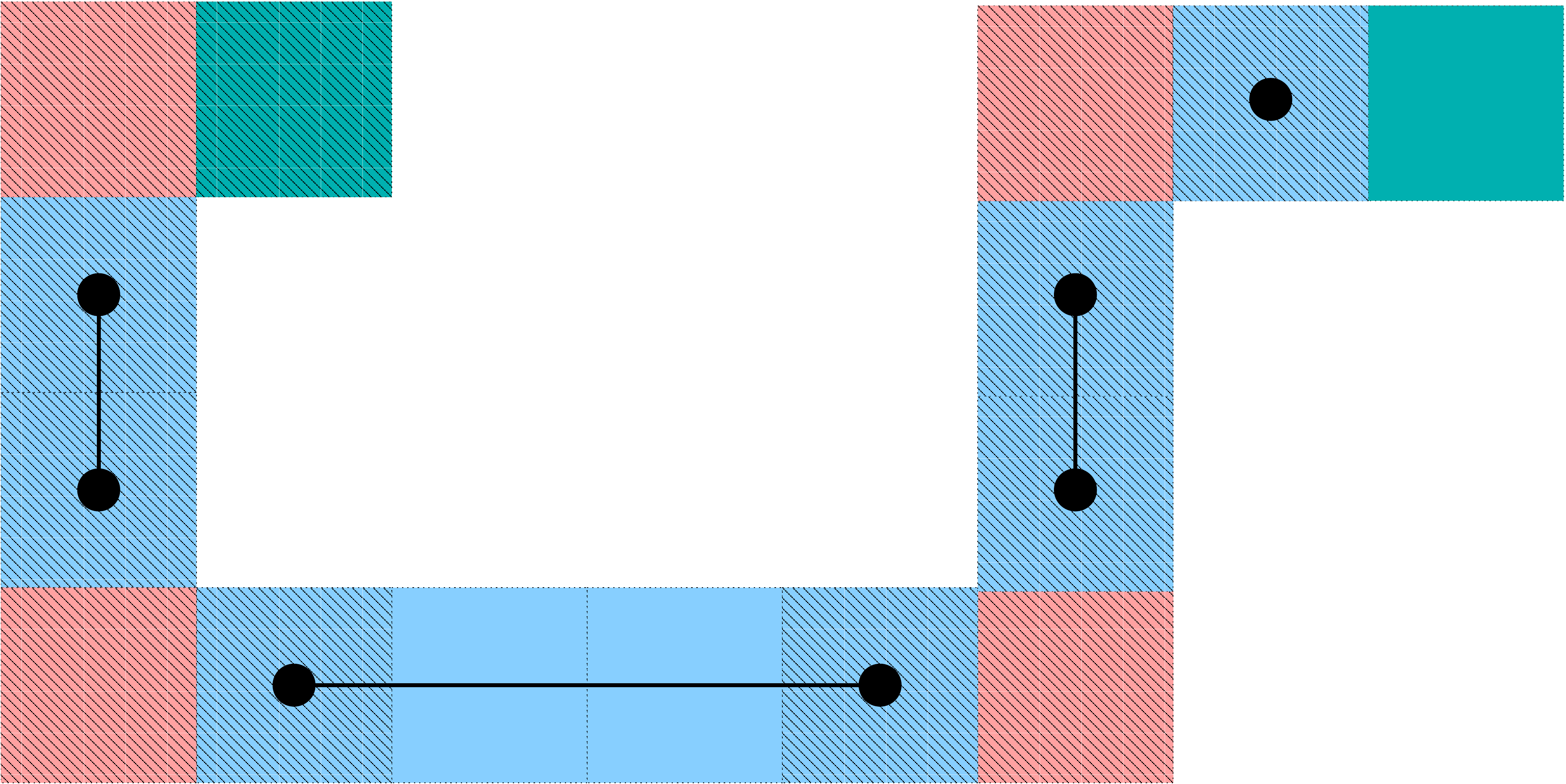} % add: grid
%\put(13,13){\tiny $D$}
%\put(45,13){\tiny $D'$}
\end{overpic}
\caption{Corner cubes in pink, $I$-cubes in light blue, and terminal cubes in cyan. Cubical neighborhoods of corners shaded and  $I$-blocks marked with graphs; only one reduced $I$-block.}
\label{fig:CubeTypes}
\end{figure}

%\note{Symmetry plane otettu pois.}
%Finally, we note that, given a corner $Q\in A$ and distinct adjacent cubes $Q'$ and $Q''$ to $Q$ in $A$, there exists a unique affine hyperplane (of codimension $1$) $P_A(Q)$ in $\R^n$ so that $P_A(Q)$ separates $Q'$ from $Q''$ and so that $|\cN_A(Q)|\setminus P_A(Q)$ consists of two components isometric to each other. We call $P_A(Q)$ the \emph{symmetry plane of $Q$ (with respect to $A$)}.

%We are now ready to restate Proposition \ref{prop:inner_bilip} using this terminology.

%{\color{blue}
%Cubical arcs are uniformly bilipschitz to segments of the same length in their %inner geometry. For a slightly more precise statement,} 
%%For the statement, 

For the statement of bilipschitz equivalence of cubical arcs, 
we introduce the notion of a symmetry plane. Let $Q$ be a cube in an arc ${\sf A}$ which is not a terminal cube. Then $\cN_{\sf A}(Q)$ consists of three cubes and there exists a unique affine hyperplane (of codimension $1$) $P_{\sf A}(Q)$ in $\R^3$ which divides $|\cN_{\sf A}(Q)|$ into two congruent $3$-cells. We call $P_{\sf A}(Q)$ the \emph{symmetry plane of $Q$ (with respect to $\sf A$)}; see Figure \ref{fig:Corner2} for a symmetry plane for a corner. We use the same terminology and notation also in the case of cubical loops.

\begin{figure}[h!]
\begin{overpic}[scale=0.2,unit=1mm]{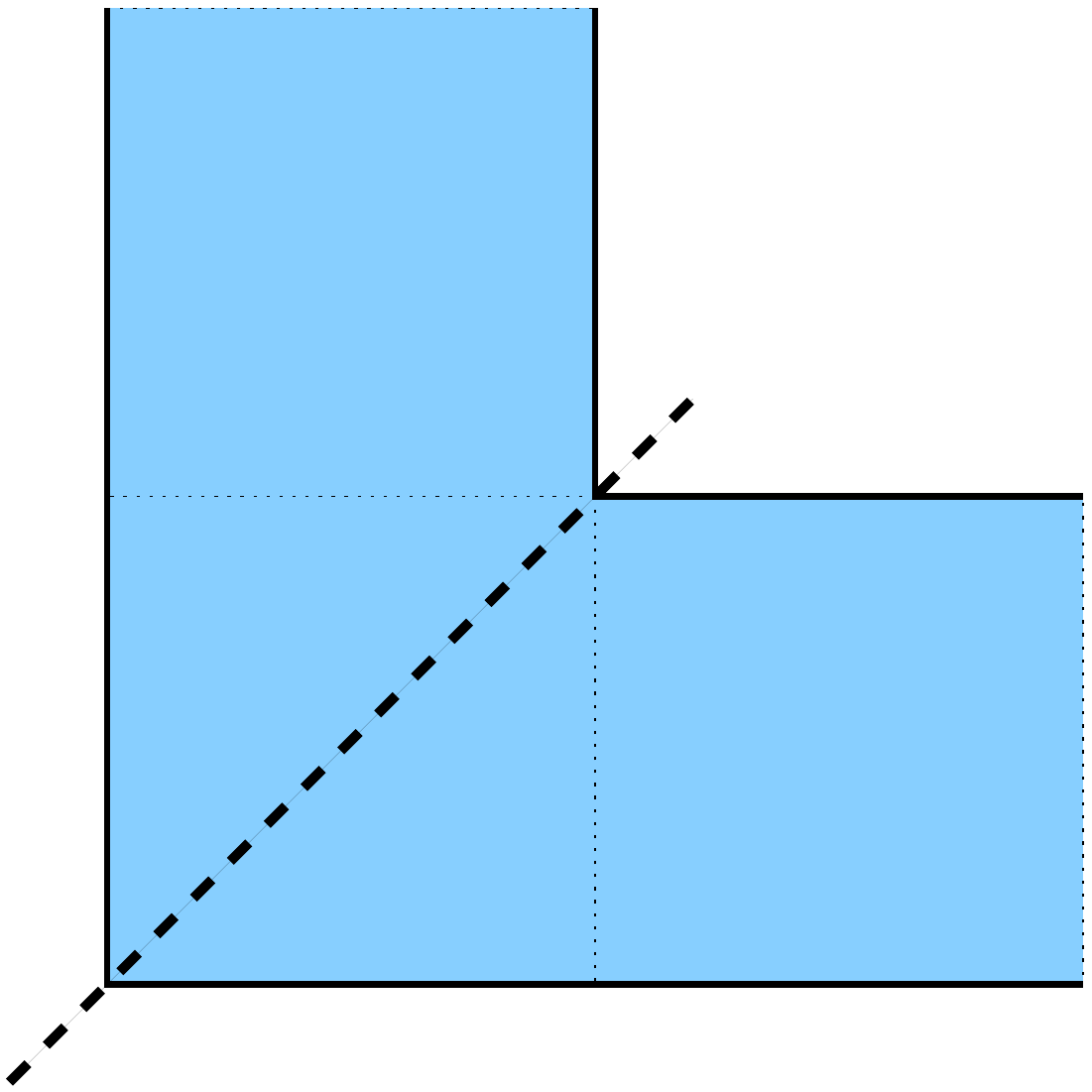} % add: grid
%\put(13,13){\tiny $D$}
%\put(45,13){\tiny $D'$}
\end{overpic}
\caption{Symmetry plane of a corner.}
\label{fig:Corner2}
\end{figure}

%\marginpar{\tiny Lemmaa tarvitaan oikeasti vain kulmien suoristamiseen.}

\begin{lemma}
\label{lemma:inner_bilip}
There exists an absolute constant $L_\inner \ge 1$ with the following property: Let $\sf A$ be an arc in $\R^3$, and $\sf S$ a segment having the same number of cubes as $\sf A$ and having cubes of the same side length than $\sf A$. Then there exists a homeomorphism $\varphi_{\sf A}^{\sf S} \colon |\sf A|\to |\sf S|$ having the following properties:
\begin{enumerate}
\item for each $Q^{\sf A}_i \in {\sf A}$, the restriction $\varphi_{\sf A}^{\sf S} \colon |\cN_{\sf A}(Q^{\sf A}_i)|\to |\cN_{\sf A}(Q^{\sf S}_i)|$ is a well-defined homeomorphism, 
\item for each reduced $I$-block $\sf B\subset \sf A$, there exists an $I$-block $\sf B'\subset \sf S$ for which $\varphi_{\sf A}^{\sf S}|_{|\sf B|} \colon |\sf B|\to |\sf B'|$ is an isometry,
\item for each corner $Q^{\sf A}_i \in \sf A$, the restriction $\varphi_{\sf A}^{\sf S}|_{|\cN(Q^{\sf A}_i)|} \colon |\cN_{\sf A}(Q^{\sf A}_i)| \to |\cN_{\sf A}(Q^{\sf S}_i)|$ is $L_\inner$-bilipschitz, and
\item for each $Q^{\sf A}_i\in {\sf A}$, which is not a terminal cube in $\sf A$, the restriction $\varphi^S_{\sf A}|_{P_{\sf A}(Q^{\sf A}_i) \cap Q^{\sf A}_i} \colon P_{\sf A}(Q^{\sf A}_i) \cap Q^{\sf A}_i \to P_{\sf S}(Q^{\sf S}_i) \cap Q^{\sf A}_i$ is well-defined and affine.
% and $\varphi_A^S$ is affine on $P_A(Q) \cap Q$.
\end{enumerate}
\end{lemma}

\begin{proof}
It suffices to consider the special case that $\#\sf A =3$ and $\sf A$ has a corner. In this case, we may assume that $|\sf A|$ is the $3$-cell $D\times [0,1]$, where
\[
D = \left([0,1]\times [1,2]\right) \cup [0,1]^2 \cup \left( [1,2]\times [0,1] \right).
\]
Let now $\sf A'$ be a straight arc of three cubes. Again, we may assume that $|{\sf A'}|=D'\times [0,1]$, where $D' = \left( [-1,0]\times [0,1]\right) \cup [0,1]^2 \cup \left([1,2]\times [0,1]\right)$; see Figure \ref{fig:Corner1}.

\begin{figure}[h!]
\begin{overpic}[scale=0.2,unit=1mm]{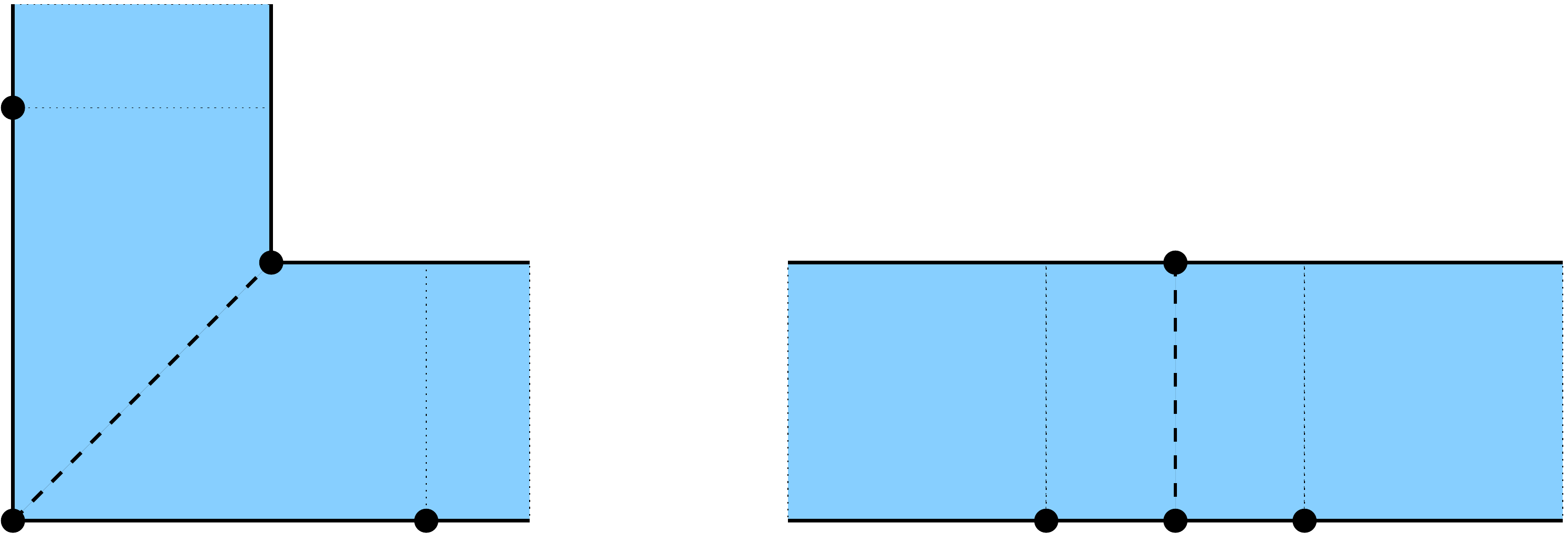} % add: grid
\put(13,13){\tiny $D$}
\put(45,13){\tiny $D'$}
\end{overpic}
\caption{Cells $D$ and $D'$}
\label{fig:Corner1}
\end{figure}

Then, clearly, there exists a bilipschitz homeomorphism $\varphi \colon D \to D'$ for which 
\begin{enumerate}
\item $\varphi|_{\{2\}\times [0,1]} = \id$ and 
%\item $\varphi(\{(t,t)\colon t\in [0,1]\}) = \{0\}\times [0,1]$, and 
\item $\varphi|_{[0,1]\times \{2\}} \colon [0,1]\times \{2\} \to \{-1\}\times [0,1]$ is an isometry. \label{item:isometry}
\end{enumerate}
Now $\varphi \times \id \colon D\times [0,1]\to D'\times [0,1]$ satisfied the requirements of the claim. 

The general case of an arc $\sf A$ now follows by straightening all neighborhoods of corners in $\sf A$ by copies of the map $\varphi$ and extending isometrically over all reduced $I$-blocks using the property \eqref{item:isometry}.  
\end{proof}

%A corresponding bilipschitz equivalence statement for cubical loops follows immediately from Lemma \ref{lemma:inner_bilip} and we state it as a corollary.

%\marginpar{\tiny Ei oikeasti tarvita.}

%\begin{corollary}
%\label{cor:loop_inner_bilip}
%There exists an absolute constant $L_\inner \ge 1$ with the following property: Let $\sL$ be a cubical loop in $\R^3$, and $\sf M$ a model loop having the same number of cubes as $\sf L$ and having cubes of the same side length than $\sf L$. Then there exists a homeomorphism $\varphi_{\sf L}^{\sM} \colon |\sL |\to |\sf M|$ having the following properties:
%\begin{enumerate}
%\item for each $Q^{\sf L}_i \in {\sf L}$, the restriction $\varphi_{\sf L}^{\sM} \colon |\cN_{\sf L}(Q^{\sf L}_i)|\to |\cN_{\sf A}(Q^{\sM}_i)|$ is a well-defined homeomorphism, 
%\item for each reduced $I$-block $\sf B\subset \sL$, there exists an $I$-block $\sf B'\subset \sf M$ for which $\varphi_\sL^{\sf M}|_{|\sf B|} \colon |\sf B|\to |\sf B'|$ is an isometry, 
%\item for each corner $Q \in \sL$, there exists a cube $Q'\in \sf M$ for which restriction $\varphi_\sL^{\sf M} |_{|\cN(Q)|} \colon |\cN(Q)| \to |\cN(Q')|$ is $L_\inner$-bilipschitz and the restriction $\varphi^{\sM}_{\sf L}|_{P_{\sf L}(Q) \cap Q} \colon P_{\sf L}(Q) \cap Q \to P_{\sf M}(Q') \cap Q'$ is well-defined and affine.
%\item for each $Q^{\sf L}_i\in {\sf L}$, the restriction $\varphi^{\sM}_{\sf L}|_{P_{\sf L}(Q^{\sf L}_i) \cap Q^{\sf L}_i} \colon P_{\sf L}(Q^{\sf L}_i) \cap Q^{\sf L}_i \to P_{\sf M}(Q^{\sf M}_i) \cap Q^{\sf L}_i$ is well-defined and affine.
%\end{enumerate}
%\end{corollary}

\section{Nested loops and twist maps}

In this section we consider pairs $(\sL,\sL')$ of cubical loops $\sL$ and $\sL'$ for which $|\sL'|\subset \interior |\sL|$ and the pair $(|\sL|, |\sL'|)$ is homeomorphic to the pair $(\bar B^2\times \bS^1, \bar B^2(1/2)\times \bS^1)$. 
We also consider particular self-homeomorphisms $\theta \colon (|\sL|,|\sL'|)\to (|\sL|,|\sL'|)$ of pairs which are identity on the boundary of $|\sL|$ and, in an heuristic sense, rotate the inner cubical loop $\sL'$. Before giving a precise definitions, we describe a round model for these self-homeomorphisms.

Let $0<r<R$ and $\ell>0$, and consider the solid tori $T(R;\ell) = \bar B^2(R)\times \bS^1(\ell)$ and $T(r;\ell) = \bar B^2(r)\times \bS^1(\ell)$. Now $(T(R;\ell),T(r;\ell))$ is a pair of solid $3$-tori homeomorphic to $(\bar B^2\times \bS^1, \bar B^2(1/2)\times \bS^1)$. Let now $\alpha \in [0,2\pi)$ be an angle. Let also $u_{r,R}\colon \bar B^2(R)\to [0,\alpha]$ be the function 
\[
x\mapsto \left\{ \begin{array}{ll}
\alpha, & |x|\le r \\
\alpha \frac{R-|x|}{R-r}, & 0\le |x|\le R\\
\end{array}\right.
\]
and $h^\alpha_{r,R;\ell} \colon T(R;\ell) \to T(R;\ell)$ be the homeomorphism
\[
(x,z) \mapsto (x, e^{iu_{r,R}(x)} z).
\]
In fact, $h^\alpha_{r,R;\ell}$ is a homeomorphism of pairs 
\[
h^\alpha_{r,R;\ell} \colon (T(R;\ell),T(r;\ell)) \to (T(R;\ell),T(r;\ell)),
\]
which is the identity on the boundary $\partial T(R;\ell)$ of $T$ and a rotation, by angle $\alpha$, on the solid torus $T(r;\ell)$. Metrically, $h^\alpha_{r,R;\ell}$ is bilipschitz with a constant $L\ge 1$ depending only on the angle $\alpha$ and the ratio $R/r$.

We describe now the cubical versions of the pair $(T(R;\ell),T(r;\ell)$ and homeomorphisms $h^\alpha_{r,R;\ell}$. We begin with nested loops and define after that the cubical twist maps.

\subsection{Nested loops}

A pair $(\sL,\sL')$ of cubical loops is a \emph{nested pair} if 
\begin{enumerate}
\item $|\sL'| \subset \interior |\sL|$, 
\item $I$-cubes of $L$ contain only $I$-cubes of $\sL'$,
\item each corner of $\sL$ contains exactly one corner of $\sL'$,
\item for each corner $Q$ of $\sL$ and the unique corner $Q'$ of $L'$ contained
in $Q$, have the same symmetry plane, that is, $P_\sL(Q)=P_{\sL'}(Q')$;
\end{enumerate}
see Figure \ref{fig:NestedLoops} for an illustration. Further, we say the nested pair $(\sL,\sL')$ is \emph{uniform} if 
\[
\dist(|\sL'|, \partial |\sL|) \ge \sigma(\sL');
\]
where $\sigma(\sL')$ is the common side length of cubes in $\sL'$.

\begin{figure}[h!]
\begin{overpic}[scale=0.1,unit=1mm]{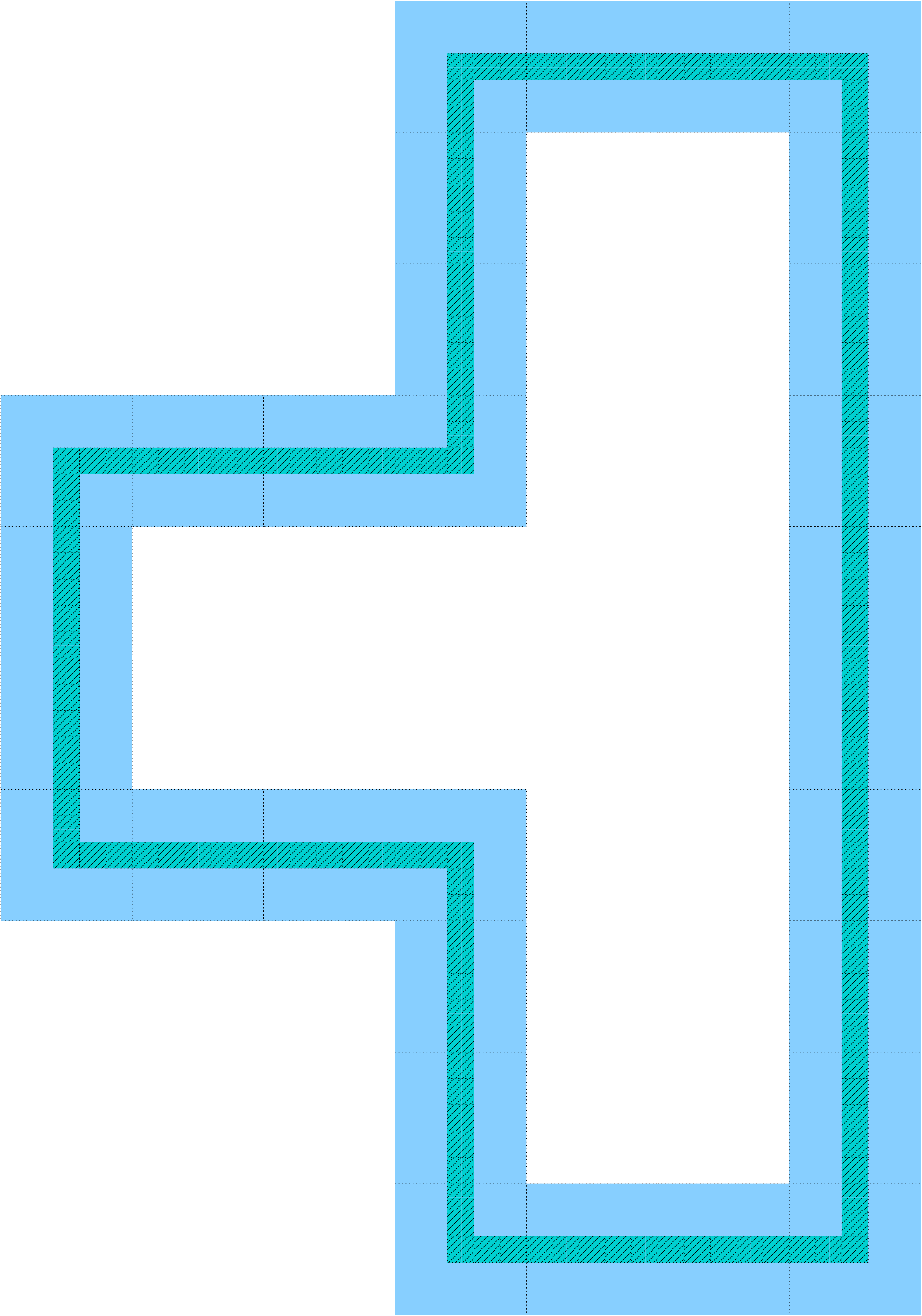} % add: grid
%\put(13,13){\tiny $D$}
%\put(45,13){\tiny $D'$}
\end{overpic}
\caption{Nested loops.}
\label{fig:NestedLoops}
\end{figure}

A nested pair $\mathscr{P}$ has a natural parametrization $\tau_{\mathscr{P}}$ with a pair of round solid tori. We record this observation as a lemma. In what follows, when we refer to a natural parametrization of a nested pair, we mean the map constructed in the proof of this lemma. We also refer to the pair $(T(R;\ell), T(r;\ell))$ of solid $3$-tori as the \emph{round model for the nested pair $\mathscr{P}$}.

\subsection{Twist maps}

We define now self-homeomorphisms of pairs $(|\sL|,|\sL'|)$ for nested cubical loops which are the basis of the our forthcoming constructions. 

Let $(\sL, \sL')$ be a pair of nested loops. Then there exists a natural map $\iota \colon \sL' \to \sL$ induced by inclusion of the cubes in $\sL'$ into cubes in $\sL$, that is, for each $Q'\in \sL'$ the image $\iota(Q')$ is the unique cube in $\sL$ containing $Q'$. 

Given an order $<_{\sL}$ for cubes in $\sL$, there exists a unique cyclic order $<_{\sL'}$ for cubes in $\sL'$ for which the map $\iota$ is order preserving. Therefore, we may consider each nested pair as an ordered pair in this sense. 
%Similarly, we may also say that a level-wise nested sequence $\mathscr{L} = ( \sL_0,\sL'_0, \sL_1, \sL'_1,\ldots, \sL_{k-1}, \sL'_{k-1}, \sL_k)$ is \emph{ordered} if each nested pair $(\sL_j,\sL'_j)$ is ordered.

Let now $\mathscr{P} = (\sL,\sL')$ be an ordered nested pair. We call an orientation preserving bijection $\rho \colon \sL' \to \sL'$ a \emph{rotation of $\sL'$ in $\sL$}. 

%Since $\sL'$ is oriented, there exists a well-defined minimal $a(\rho)\in \N$ which is, for each $Q'\in \sL'$, the distance of $\rho(Q')$ and $Q'$ in the order of $\sL'$. We call the number $2\pi a(\rho)/\#\sL'\in [0,2\pi[$ the \emph{rotation angle of $\rho$}.

\begin{definition}
\label{def:twist}
A homeomorphism $h \colon (|\sL|, |\sL'|) \to (|\sL|, |\sL'|)$ is a \emph{$\rho$-twist} if $h|_{\partial |\sL|} = \id$ and, for each $Q'\in \sL'$, 
\begin{enumerate}
\item $h(\cN_{\sL'}(Q')) = \cN_{\sL'}(\rho(Q'))$, and
\item if $Q'$ and $\rho(Q')$ are not corners or adjacent to corners of $\sL'$, then $h|_{Q'}\colon Q'\to \rho(Q')$ is an isometry.
\end{enumerate}
\end{definition}

The main observation of this section is the following lemma. For the statement, let $R_\pm \colon \R^3 \to \R^3$ be the linear involution $(x_1,x_2,x_3) \mapsto (x_1,\pm x_2, -x_3)$, as in the introduction. We also say that a nested pair $\mathscr P = (\sL,\sL')$ is \emph{$R_\pm$-symmetric} if the maps $\sL \to \sL$, $Q\mapsto R_\pm(Q)$, and $\sL'\to \sL'$, $Q'\mapsto R_\pm(Q')$, are well-defined  bijections.

%\marginpar{\tiny Tarvitaan!}
\begin{lemma}\label{lemmatarkea}
There exists a constant $L_\Diamond\ge 1$ with the following property: Let $\mathscr P =(\sL, \sL')$ be a uniform nested pair symmetric with respect to $R_\pm$ and $\rho \colon \sL'\to \sL'$ a rotation. Then there exists a $\rho$-twist $h_\rho \colon (|\sL|, |\sL'|) \to (|\sL|, |\sL'|)$ for which the involution 
\begin{equation}\label{eq:cancellation}
h_\rho \circ R_\pm \circ h_\rho^{-1} \colon (|\sL|, |\sL'|) \to (|\sL|, |\sL'|)
\end{equation}
is locally $L_\Diamond \frac{\sigma(\sL)}{\sigma(\sL')} (\#\sL)$-bilipschitz.
\end{lemma}

\begin{proof}
By the method of Lemma \ref{lemma:inner_bilip}, there exists an absolute constant $L\ge 1$ and an $L$-bilipschitz homeomorphism $h'_\rho \colon |\sL'| \to |\sL'|$ realizing the rotation $\rho$, that is, a map satisfying conditions (1) and (2) in Definition \ref{def:twist} for each cube in $\sL'$. We extend now $h'_\rho$ to a $\rho$-twist $h_\rho \colon |\sL|\to |\sL|$.

Let $\widetilde \sL$ be the cubical line $\widetilde \sL = \{ \sigma(\sL) ( k e_1 + [0,1]^3) \colon k \in \Z\}$ in $\R^3$. We fix in $\widetilde \sL$ the natural order induced by $\Z$ and let $\phi \colon \widetilde \sL \to \sL$ be a natural cubical covering map, that is, an order preserving map $\Gamma(\widetilde \sL) \to \Gamma(\sL)$ which is an injection in $\cN_{\widetilde{\sL}}(\tilde Q)$ for each $\tilde Q\in \widetilde \sL$.

Using the method of Lemma \ref{lemma:inner_bilip} again, we may now fix a locally $L$-bilipschitz covering map $g_\phi \colon |\widetilde \sL|\to |\sL|$ realizing $\phi$, that is, satisfying (1) and (2)  in Definition \ref{def:twist} for each cube in $\widetilde \sL$. Note that, we may assume that cubes in $|\sL|$ have identical lifts, that is, if $\tilde Q_1$ and $\tilde Q_2$ in $\sL$ satisfy $g_\phi(\cN_{\widetilde{\sL}}(\tilde Q_1)) = g_\phi(\cN_{\widetilde{\sL}}(\tilde Q_2))$ then $g_\phi(x + ke_1) = g_\phi(x)$ for $x\in \tilde Q_1$, where $k\in \Z$ satisfies $\tilde Q_2 = \tilde Q_1 +ke_1$.

Since $(\sL, \sL')$ is a nested pair, we have by condition (2) in the definition of nested pairs, that we may in addtion assume that there exists a vector $v\in \R^3$ for which the infinite cubical line $\widetilde{\sL'} = \{ \sigma(\sL)( ke_1 + [0,1]^k + v) \colon k \in \Z\}$ has the property that there exists a natural cubical covering map $\phi' \colon \widetilde{\sL'} \to \sL'$ for which the restriction $g_\phi|_{|\widetilde{\sL'}|}$ is a realization of $\phi'$, that is, 
\[
g_\phi(\cN_{\widetilde{\sL'}}(\tilde Q')) = \cN_{\widetilde{\sL'}}(\phi'(\tilde Q'))
\]
for each $\tilde Q' \in \widetilde{\sL'}$.

Let now $\tilde h'_\rho \colon |\widetilde {\sL'}|\to |\widetilde{\sL'}|$ be a lift of $h'_\rho$ in the covering map $g_\phi$. To define an extension of $\tilde h'_\rho$ we give preliminary definitions.

We observe first that $|\widetilde \sL| = \R\times [0,\sigma(\sL)]^2$ and $|\widetilde{\sL'}| = \R\times (v+[0,\sigma(\sL')]^2)$, where $[0,\sigma(\sL)]^2$ and $v+[0,\sigma(\sL')]^2$ a concentric squares in $\R^2$. Let now $s \colon [0,\sigma(\sL)]^2 \to v+[0,\sigma(\sL')]^2$ be the unique square scaling mapping $\widetilde{\sL} \to \widetilde{\sL'}$, and let $S = \id \times s \colon |\sL|\to |\sL'|$. Then $S$ is a bilipschitz homeomorphism with a constant depending only on the ratio $\sigma(\sL)/\sigma(\sL')$.

We fix now a lift of the identity map $\id \colon \partial |\sL|\to \partial |\sL|$ as follows. Let $\tilde Q\in \sL$ be cube for which $\phi(\tilde Q)$ is an $I$-cube, and let $x_0\in \tilde Q\cap \partial |\widetilde \sL|$. Then there exists a unique lift $\iota \colon \partial |\widetilde \sL| \to \partial |\widetilde \sL|$ of the identity map $\id \colon \partial |\sL|\to \partial |\sL|$ satisfying $S(\iota(x_0))=\tilde h'_\rho(S(x_0))$.  

We are now ready to define the extension of $\tilde h'_\rho$. Let $\tilde h_\rho \colon |\widetilde \sL|\to |\widetilde \sL|$ be the unique map satisfying  $\tilde h_\rho|_{\partial |\widetilde \sL|} = \iota$, $\tilde h_\rho|_{|\widetilde \sL'} = \tilde h'_\rho$, and which, for each $x\in \partial |\widetilde \sL|$ affinely maps the segment $[x,S(x)]$ to the segment $[\iota(x),\tilde h'_\rho(S(x))]$.

To estimate the bilipschitz constant of $\tilde h_\rho$, we observe that, since $\tilde h'_\rho$ is a lift of a realization of a rotation $\rho$, there exists a bijection $\tilde \rho \colon \widetilde{\sL'}\to \widetilde{\sL'}$ having $\tilde h'_\rho$ as its realization. In fact, $\tilde \rho$ is a lift of $\rho$. The bijection $\tilde \rho$ is a translation in $\widetilde{\sL'}$ by $k_\rho\in \Z$ cubes, where $k_\rho$ is the rotation distance of $\rho$ in $\sL'$. Thus, the bilipschitz constant of $\tilde h_\rho$ depends only on the ratio of the distance $\dist(\partial |\sL|, |\sL'|)$ and the translation distance $(\sigma(\sL') k_\rho$ of $\rho$. Since the translation distance is at most the length of the cubical torus $|\sL|$ and the pair $(\sL,\sL')$ is uniform, we conclude that the bilipschitz constant of $\tilde h_\rho$ is bounded by the ratio
\[
\frac{\sigma(\sL')k_\rho}{\dist(\partial |\sL|, |\sL'|)} \lesssim \frac{\sigma(\sL) \#\sL}{\sigma(\sL')} = \frac{\sigma(\sL)}{\sigma(\sL')} \# \sL.
\]

Since both $\iota$ and $\tilde h'_\rho$ are deck transformation for restrictions of $g_\phi$ and the covering map $g_\rho$ is natural translation invariance, we conclude that $\tilde h_\rho$ is a deck transformation of the covering map $g_\phi$, that is, $g_\phi = g_\phi \circ \tilde h_\rho$. Thus $\tilde h_\rho$ decends to a bilipschitz homeomorphism $h_\rho \colon |\sL|\to |\sL|$. This completes the construction of $\tilde h_\rho$. It remains to show the local bilipschitz estimate \eqref{eq:cancellation}.
 
Let $R\colon \R^3 \to \R^3$ be the involution $R_+$ and suppose that the nested pair $(\sL,\sL')$ is symmetric with respect to $R$; the case of $R_-$ is similar. Then  $\{0\}\times \R^2 \cap |\sL|$ is the fixed point set of $R|_{|\sL|}$.

Since $|\sL|$ is $R$-symmetric, we observe that also the map $g_\phi$ is symmetric in the sense that $g_\phi(-t,w) = R(g_\phi(t,w))$ for all $(t,w) \in \R \times [0,\sigma(\sL)]^2$. Thus, the lift $\tilde R\colon |\widetilde \sL|\to |\widetilde \sL|$ of $R|_{|\sL|} \colon |\sL|\to |\sL|$ under $g_\phi$ is actually the involution $(x,y,z) \mapsto (-x,y,z)$. 

We conclude that 
\[
\tilde H = \tilde h_\rho \circ \tilde R \circ \tilde h_\rho^{-1} \colon |\widetilde \sL|\to |\widetilde \sL|
\]
is therefore an involution with the following properties:
\begin{enumerate}
\item on $\partial |\widetilde \sL|$ is the map $(t,z) \mapsto \iota(-\iota_1(t,x), \iota_2(t,w))$, where $\iota = (\iota_1,\iota_2) \colon |\widetilde \sL| \to \R\times [0,\sigma(\sL)]^2$,
\item on $|\widetilde{\sL'}|$ is the map $(t,w) \mapsto \tilde h'_\rho(-(\tilde h'_\rho)_1(t,w), (\tilde h'_\rho)_2(t,w))$, where $\tilde h'_\rho = ((\tilde h'_\rho)_1, (\tilde h'_\rho)_2) \colon |\widetilde{\sL'}| \to \R\times v+[0,\sigma(\sL')]^2$, and
\item for each $p\in \partial |\sL|$, affine map on the segment $[p,S(p)]$.
\end{enumerate}

For points $p\in \partial |\widetilde \sL|$, we further have the estimate 
\[
|\tilde H(p)-\tilde H(S(p))| \le 2 \sigma(\sL')k_\rho.
\]
Thus $\tilde H$ is bilipschitz with a depending only on  
\[
\frac{\sigma(\sL)}{\sigma(\sL')} \# L.
\]

Since 
\[
g_\phi \circ \tilde h_\rho \circ \tilde R \circ \tilde h_\rho^{-1} = h_\rho \circ g_\phi \circ \tilde R \circ \tilde h_\rho^{-1} = h_\rho \circ R \circ g_\phi \circ \tilde h_\rho^{-1} = h_\rho \circ R \circ h_\rho^{-1} \circ g_\phi,
\]
we conclude that $h_\rho \circ R \circ h_\rho^{-1}$ satisfies the bilipschitz estimate \eqref{eq:cancellation}.
\end{proof}

\section{Bing's wild involution}

We begin now the construction of a cubical version of Bing's wild involution. The involution is based on the construction of a defining sequence for a decomposition in $\bS^3$ and shrinkability of this decomposition. The wild Cantor set obtained by this construction is called Bing's double. We recall first this topological part of the construction and then discuss the involution. 

We refer to Daverman's book \cite{Daverman-book} for the terminology related to decomposition spaces. We merely recall that the decomposition space $\R^3/E$ associated to a compact set $E\subset \R^3$ is the quotient space $\R^3/{\sim}$, where $\sim$ is the minimal equivalence relation for which $x\sim y$ if and only if the points $x$ and $y$ in $\R^3$ belong to the same component of $E$.

\begin{remark}
Although the statement of Theorem \ref{thm:main} is for involutions of $\bS^3$, we work in $\R^3$ to simplify the notation. Note that the constructions of the defining sequences take place in a solid $3$-torus in $\R^3$ and hence the wild involutions $\R^3\to \R^3$ we construct, extend naturally to involutions of $\bS^3$ fixing the north pole $e_4$ after identification of $\R^3$ with $\bS^3\setminus \{e_4\}$ by the stereographic projection.
\end{remark}

\subsection{The Set-up}

For the rest of the discussion, we fix $R \colon \R^3\to \R^3$ to be either the orientation reversing involution 
\[
R \colon (x,y,z) \mapsto (-x,y,z)
\]
considered by Bing \cite{Bing} or the orientation preserving involution 
\[
R \colon (x,y,z)\mapsto (-x,-y,z)
\]
considered by Montgomery and Zippin \cite{Montgomery-Zippin}. Having this choice in mind, we fix now a decomposition $\sB$ of $\R^3$, called Bing's double, which is invariant under this involution.

The defining sequence $(X_k)$ for the decomposition $\sB$ is given as follows. First, let $(\sL_w)_w$ be a tree of cubical loops as in \cite{Bing, Bing1988}, where $w$ is a finite word in letters $\{+,-\}$; the cubical loop corresponding the empty word $\emptyset$ we denote $\sL_0$. We refer to \cite{Bing, Bing1988} for a description of the linking; the tori are said to 'hook elbows'. As in \cite{Bing, Bing1988} and \cite{Montgomery-Zippin}, we assume that each cubical loop $\sL_w$ is symmetric with respect to the fixed involution $R$, that is, for each word $w$, the map $\sL_w \to \sL_w$, $Q \mapsto R(Q)$, is a well-defined bijection. In particular, $R(|\sL_w|) = |\sL_w|$ for each word $w$.

%For the defining sequence for the decomposition, we consider a tree of cubical loops $(\sL_w)_w$, where $w$ is a finite word in letters $\{+,-\}$; the cubical loop corresponding the empty word $\emptyset$ we denote $\sL_0$. 

%For each word $w$, let $\sA^+_w$ and $\sA^-_w$ be arcs splitting $\sL_w$ and let $\sL_{w+}$ and $\sL_{w-}$ be arcs properly embedded into $\sA^+_w$ and $\sA^-_w$ in such a way that the  solid $3$-tori $|\sL_{w+}|$ and $|\sL_{w-}|$ are linked in $|\sL_w|$ but not in $\R^3$. We refer to \cite{Bing, Bing1988} for a description of the linking; the tori are said to 'hook elbows'.

For each $k\in \N$, we set $X_k$ to be the union of all cubical loops $\sL_w$ for words $w$ of length $k$. Let now $\sB$ be the decomposition of $\R^3$ for which the non-trivial elements are the components of the intersection
\[
X=\bigcap_{k \ge 1} X_k.
\]

%The associated decomposition of $\R^3$ is now given by the components in the intersection 
%\[
%\sB = \bigcap_{k \ge 1} \bigcup_{|w|=k} |\sL_w|,
%\]
%where $|w|$ is the length of the word $w$.

Bing shows the shrinkability of the decomposition $\sB$ by a delicate folding procedure of the tori $|\sL_w|$; see \cite[Figures 9--12]{Bing1988}. Using our terminology, we may say that these foldings corresponds to a choice of a family of twist maps $\{ |\sL_w|\to |\sL_w| \}_w$. The particular choices of the angles of associated twist maps is not relevant for our considerations. Thus, in what follows, we discuss mainly the properties of the these twist maps and refer to \cite[pp.~491--492]{Bing1988} for a careful description of the foldings which yield the shrinkability of the decomposition $\sB$. 

Let $\phi \colon \R^3 \to \R^3$ be a monotone map associated to the decomposition $\sB$, that is, $\phi|_{\R^3\setminus X} \colon \R^3 \setminus X \to \R^3\setminus \phi(X)$ is a homeomorphism and $\phi(C)$ is a point for each non-trivial element $C$ of $\sB$. We denote $f_\phi \colon \R^3\to \R^3$ the (wild) involution induced by $\phi$, that is, $f_\phi \circ \phi = \phi \circ R$. The purpose of the following sections is to show that, for each $p\in [1,2)$, we may choose such a monotone map $\phi$ that the involution $f_\phi$ belongs to the Sobolev space $W^{1,p}$.

\subsection{Description of the cubical properties of the configuration}

We fix now a particular cubical configuration of loops. Apart from terminology, the following conditions are as in \cite{Bing,Bing1988}.

\subsubsection{Splitting a loop}

In the construction of Bing's double, two cubical loops are embedded into a cubical loop in each stage. These new loops are properly embedded into arcs splitting the larger loop. For this reason we introduce {this} last bit of terminology.

A loop $\sL'$ is \emph{properly nested into an arc $\sA$} if 
\begin{enumerate}
\item $|\sL'| \subset \interior |\sA|$, 
\item $I$-cubes of $\sA$ contain only $I$-cubes of $\sL'$,
\item each terminal cube of $\sA$ contains exactly two corners of $\sL'$, and
\item each corner of $\sA$ contains exactly two corners of $\sL'$ and these three corners have the same symmetry plane.
\end{enumerate}
%\note{Luultavasti koko symmetriataso-juttua ei tarvita.}

Let $\sL$ be a cubical loop. A pair $\sA^+$ and $\sA^-$ of arcs is \emph{splitting of $\sL$} if $\sA^-\cup \sA^+ = \sL$ and $\sA^-\cap \sA^+ = \Terminal(\sA^-)$; note that $\Terminal(\sA^+) = \Terminal(\sA^-)$; see Figure \ref{fig:Splitting} for an illustration. 

\begin{figure}[h!]
\begin{overpic}[scale=0.3,unit=1mm]{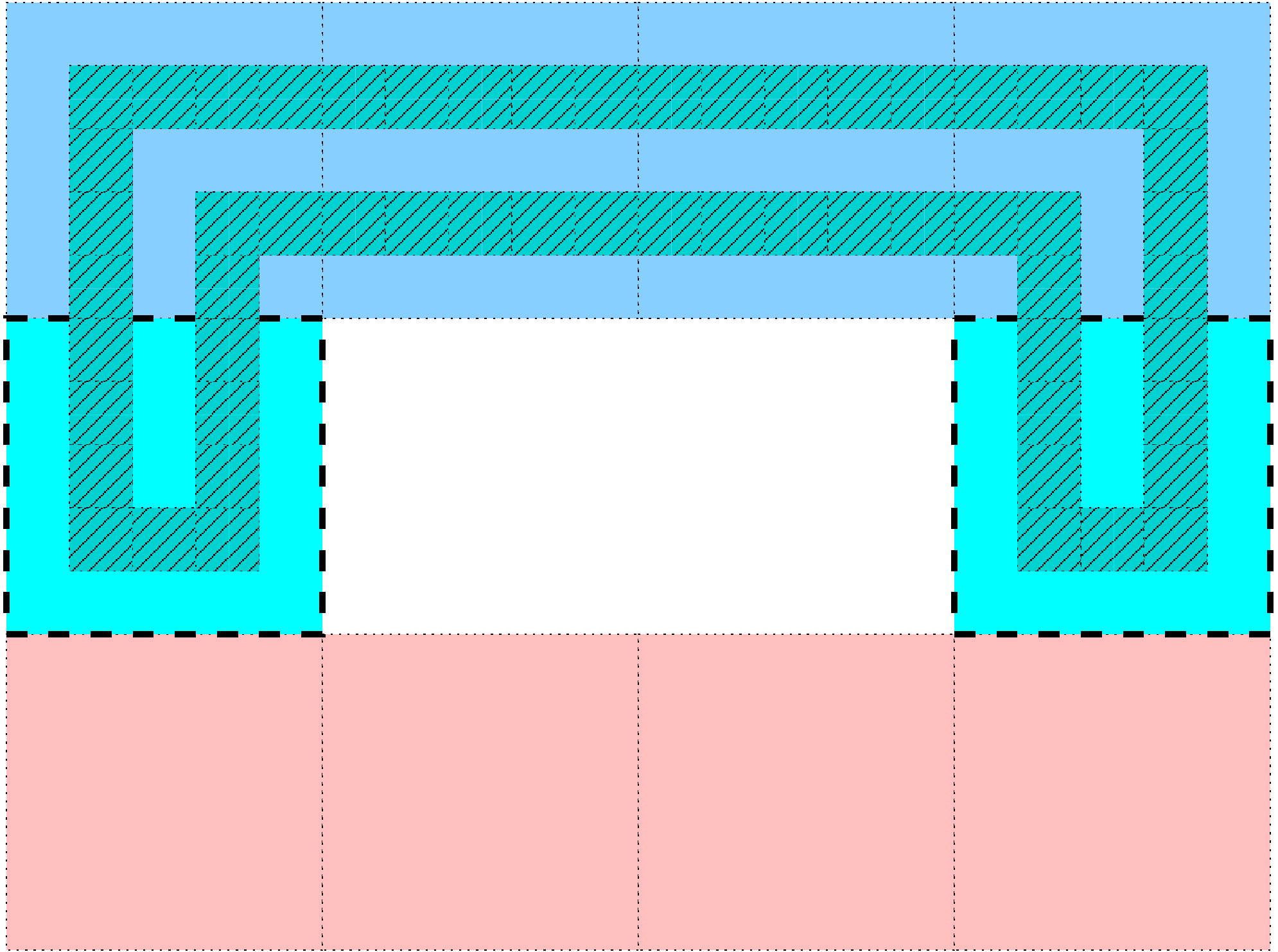} % add: grid
%\put(13,13){\tiny $D$}
%\put(45,13){\tiny $D'$}
\end{overpic}
\caption{Splitting of loops.}
\label{fig:Splitting}
\end{figure}

\subsubsection{Cubical loops}

Let $(\sL_w)_w$ be a tree of nested loops, which are invariant under $R$, as in the initial configuration of Bing's double. Let $(r_k)$ be a positive strictly decreasing sequence tending to zero to be fixed later. For each $k$, let $r_k$ be the side length of cubes in loops $\sL_w$ for all words $w$ of length $k$. The forthcoming conditions yield a rate $r_k \lesssim 15^{-k}$ for this sequence.

For each word $w$, let $\sA^+_w$ and $\sA^-_w$ be arcs splitting $\sL_w$ and let $\sL_{w+}$ and $\sL_{w-}$ be loops properly embedded into $\sA^+_w$ and $\sA^-_w$ in such a way that the  solid $3$-tori $|\sL_{w+}|$ and $|\sL_{w-}|$ are linked in $|\sL_w|$ but not in $\R^3$. We may assume that arcs $\sA^\pm_w$ are invariant under $R$ in the sense that $\sA^\pm_w \to \sA^\pm_w$, $Q\mapsto R(Q)$, is a well-defined bijection.
We may also assume that $\sL_0$ is a model loop having two long sides and two short sides and $\# \sL_0 \ge 12$, say; see Figure \ref{fig:Bing-one-level}.

For each word $w$, let $\sL'_w$ be a loop nested in $\sL_w$ of side length $s(\sL'_w) = s(\sL_w)/3$ and satisfying $\dist(|\sL'_w|, \partial |\sL_w|) = s(\sL'_w)$. We split $\sL'_w$ into two arcs $\sA^+_w$ and $\sA^-_w$ for which the terminal cubes of $\sA^+_w$ and $\sA^-_w$ are contained in the terminal cubes of the arcs in the previous level, that is, $\sA'_{w1} \cap \sA'_{w2} \subset \sA'_{v1} \cap \sA'_{v2}$, where $v$ is the unique word satisfying either $w=v1$ or $w=v2$. We also place loops $\sL_{w+}$ and $\sL_{w-}$, contained in $|\sL_w|$, into arcs $\sA^+_w$ and $\sA^-_w$ so that they are properly nested $\sA^+_w$ and $\sA^-_w$, respectively. We also require that $s(\sL_{w\pm}) \le s(\sL'_w)/5$.

\begin{figure}[h!]
\begin{overpic}[scale=0.15,unit=1mm]{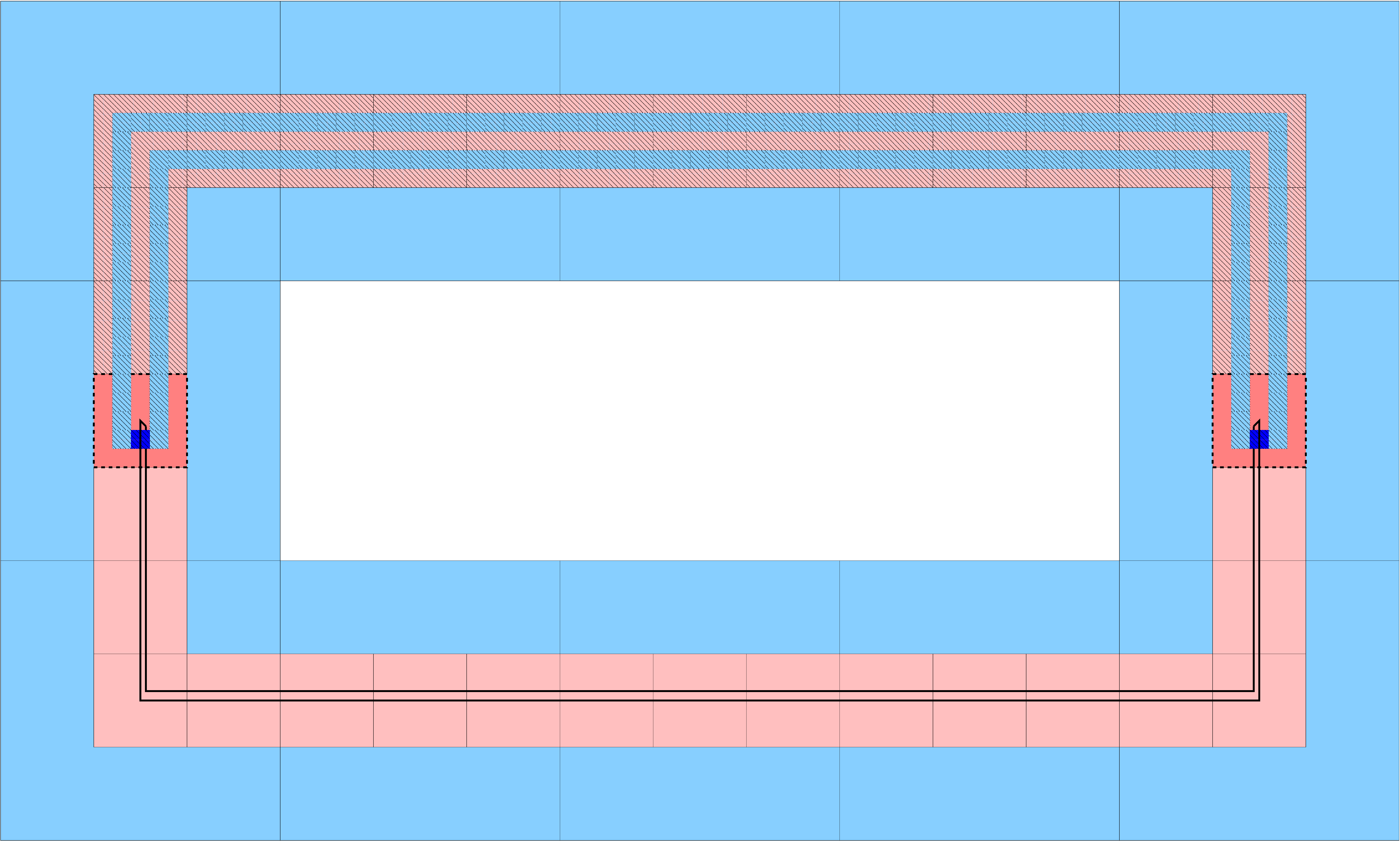} % add: grid
\put(106,50){\tiny $|\sA_+|$}
\put(106,18){\tiny $|\sA_-|$}
\end{overpic}
\caption{Nested loops $\sL'_0$ and $\sL_0$ and loops $\sL_+$ and $\sL_-$ properly embedded in arcs $\sA_+$ and $\sA_-$, respectively, splitting the loop $\sL'_0$.}
\label{fig:Bing-one-level}
\end{figure}

%{\color{red} For each word $w$, we fix an $R$-invariant model loop $\sM_w$ for $\sL_w$ having the same side length of cubes, that is, $s(\sM_w) = s(\sL_w)$, and an $R$-equivariant isomorphism $\alpha_w \colon \Gamma(\sL_w) \to \Gamma(\sM_w)$, that is, $R \circ \alpha_w = \alpha_w \circ R$; see Figure \ref{fig:Bing-first-model}. Let also $\sM'_w$ be a properly nested $R$-invariant loop in $\sM_w$ which is isomorphic to $\sL'_w$ and has the same side length as $\sL'_w$.}

\begin{figure}[h!]
\begin{overpic}[scale=0.13,unit=1mm]{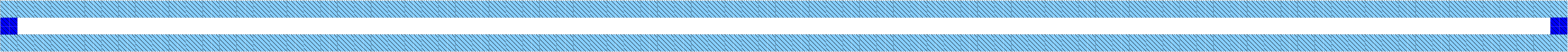} % add: grid
%\put(113,40){$|\sL_0|$}

%\put(1,64){\tiny $Q_1$}
%\put(1,2){\tiny $Q_2$}
%\put(108,2){\tiny $Q_3$}
%\put(108,64){\tiny $Q_4$}

%\put(106,50){\tiny $|\sA_+|$}
%\put(106,18){\tiny $|\sA_-|$}
%\put(100,40){\tiny $|\sL_1|$}
%\put(100,20){\tiny $|\sL_2|$}
\end{overpic}
\caption{Model loop $\sM_+$ for $\sL_+$ for in Figure \ref{fig:Bing-one-level} drawn in a slightly smaller scale.}
\label{fig:Bing-first-model}
\end{figure}

\subsubsection{The monotone map}

The monotone map $\phi \colon \R^3\to \R^3$ is defined as follows. Let $\{ \theta_w \colon |\sL_w| \to |\sL_w|\}_w$ be a family of twist maps realizing Bing's foldings so that, for each word $w$, the restriction $\theta_w|_{|\sL'_w|}$ is a realization of a shift map $\sL'_w \to \sL'_w$. 

For each $k \in \N$, let $\theta_k \colon X_k \to X_k$ be the homeomorphism satisfying $\theta_k|_{\R^3 \setminus X_k} = \id$ and $\theta_k|_{|\sL_w|} = \theta_w$ for each $w$ of length $k$. We define  $\phi \colon \R^3 \to \R^3$ by 
\[
\phi = \lim_{k \to \infty} \theta_1 \circ \cdots \circ \theta_k.
\]
Therefore, $\phi$ is a monotone map as a uniform limit.
The wild involution $f\colon \R^3 \to \R^3$ is now the unique involution satisfying
\[
f \circ \phi = \phi \circ R.
\]

To simplify notation, we write
\[
\psi_k = \theta_1 \circ \cdots \circ \theta_k \colon \R^3 \to \R^3
\]
and 
\[
f_k = \psi_k \circ R \circ \psi_k^{-1} \colon \R^3 \to \R^3
\]
for each $k\in \N$. Then $f_k \to f$ uniformly as $k\to \infty$. Note that, locally, each $\psi_k$ is a composition of twist maps and hence $\psi_k$ is a bilipschitz homeomorphism for each $k\in \N$. 

Since also $\phi$ is locally in $\R^3\setminus X$ a finite composition of twist maps, there exists a family $(\sT_w)_w$ of nested cubical tori for which $|\sT_w| = \psi_k(|\sL_w|)$ and isomorphisms $\beta_w \colon \sL_w \to \sT_w$ for which $f_{|w|} \circ \beta_w = \beta_w \circ R$ for each word $w$. More precisely, for each $w$ there exists a cubical loop $\sT'_w$ properly nested in $\sT_w$ for which $|\sT_w|\setminus |\sT'_w| = \psi_k(|\sL_w|\setminus |\sL'_w|)$ and the loops $\sT_{w+}$ and $\sT_{w-}$ are nested in arcs $\sA^{\sT}_{w+}$ and $\sA^{\sT}_{w-}$ splitting $\sT_w$. 

\subsubsection{Local bilipschitz constant of  $f_k$ on $|\sT_w| \setminus |\sT'_w|$ for $|w|=k$}

%Let $k\in \N$ and $w$ a word of length $k$. Then there exists an $L_0$-bilipschitz homeomorphism $g_w \colon |\sM_w|\to |\sT_w|$ for which the restriction $g_w|_{|\sM'_w|} \colon |\sM'_w|\to |\sT'_w|$ is a well-defined homeomorphism. We may also assume that both $g_w$ and the restriction $g_w|_{|\sM'_w|}$ are induced by isomorphisms $\Gamma(\sM_w) \to \Gamma(\sT_w)$ and $\Gamma(\sM'_w) \to \Gamma(\sT'_w)$ and that the conjugation 
%\[
%h_w = g_w^{-1} \circ f_k \circ g_w \colon |\sM_w| \to |\sM_w|
%\]
%is a twisted linear involution $h_w = h_{\vartheta_w,R}$ of $\sM_w$ for a twist map $\vartheta_w \colon |\sM_w|\to |\sM_w|$ and the chosen involution $R$. Thus, by Lemma \ref{lemma:cancellation}, we have that $h_w$, and further $f_k|_{|\sT_w|}$ are $L_1/s(\sL_w)$-bilipschitz, where $L_1\ge 1$ is an absolute constant. Note that, we use here the facts that $s(\sL_w) = s(\sT_w) = s(\sM_w)$ and $s(\sL'_w) = s(\sT'_w) = s(\sM'_w)$ for each word $w$.

%\subsubsection{Local bilipschitz constant of $f$}

We define the \emph{corner index $\eta(x)$ for a point $x\in \R^3$ with respect to the family $(\sT_w)_w$} by
\[
\eta(x) = \# \{ w \colon x \in \Corner(\sT_w)\} \in \N \cup \{\infty\}.
\]

Let now $x\in |\sT_w| \setminus |\sT'_w|$. Since $\psi_{k-1}$ is composition of twist maps, we conclude that its  local Lipschitz constant at $x$ depends only on the number of corners $\eta(x)$. More precisely, by the proof of Lemma \ref{lemma:inner_bilip}, we have $\Lip(\psi_{k-1}) \le L_{\inner}^{\eta(x)}$, where $L_{\inner}$ is universal. On the other hand, by Lemma \ref{lemmatarkea}, $\Lip(\theta_k \circ \R \circ \theta_k^{-1})  \le L_\Diamond \cdot 3 \cdot \# \sT_w$. Thus
\begin{equation}\label{annanumero}
\Lip f_k  = \Lip   (\psi_{k-1} \circ (\theta_k   \circ R \circ \theta_k^{-1}) \circ \psi_{k-1} ^{-1})     \le L_1 L_{\inner}^{2\eta(x)}/r_{\ell(x)}
\end{equation}
near $x$ where $L_1$ is universal.

%Clearly, for each $x\in |\sT_w|\setminus |\sT'_w|$ the corner index $\eta(x)$ is less than $|w|$. We note that, by \cite{Bing1988}, the number of corners of $\sT_w$ is at most double the corners of $\sL_w$ and the cubical loop $\sL_w$ has $4^{|w|}$ corners. Note that $\sT_w$ has approximately $15^{|w|}$ cubes in total.

%We estimate now the local bilipschitz constant of the wild involution $f\colon \R^3 \to \R^3$ on $|\sT_0|\setminus \phi(X)$. Note that $\phi(X) \subset |\sT_0|$ is a Cantor set of Lebesgue $3$-measure zero and the map $\phi$ is the linear involution $R$ in the complement of $|\sT_0|=|\sL_0|$.

%Let $x\in |\sT_0|\setminus \phi(X)$. Then there exists a unique index $\ell(x)\in \N$ for which $x\in X_{\ell(x)} \setminus X_{\ell(x)+1}$. Thus, near $x$, the compositon 
%\[
%f(x) = \lim_{k \to \infty} \theta_1 \circ \cdots \circ \theta_k(x)
%\]
%consists of at most $\eta(x)$ $L_2$-bilipschitz maps, which stem from straightening of the corners in maps $|\sM_w|\to |\sT_w|$, one uniformly $L_1/r_{\ell(x)}$-bilipschiz map, which is a conjugation of a twisted linear involution by fixed maps $|\sM_w|\to |\sT_w|$, and identity mappings. Thus the local bilipschitz constant $\Lip f$ of $f$ at $x$ satisfies
%\[
%(\Lip f)(x) \le L_2^{\eta(x)} L_1/r_{\ell(x)},
%\] 
%where $L_1\ge 1$ and $L_2\ge 1$ are absolute constants. Recall that $r_k$ is the common side length of cubes in all cubical loops $\sL_w$, $\sT_w$, and $\sM_w$ for $|w|=k$.

\section{Proof of Theorem \ref{thm:main}}

Let $p\in [1,2)$. It remains to show that we may choose side lengths of cubes in the loop $\sL_w$ in such a way that the reflection $f \colon \R^3 \to \R^3$ induced by the reflection $R$ is in the Sobolev space $W^{1,p}(\R^3,\R^3)$. 

For each $k\in \N$, let 
\[
m_k = \frac{k}{2-p}
\]
and, for each word $w$ of length $k$, we set the cube size $r_k$ of $\sL_w$ to be 
\[
r_k = \min\{ 3^{-m_k}, 15^{-k}, (10L_2^p)^{-k}\}.
\]

Let now the monotone map $\phi \colon \R^3 \to \R^3$ and the involution $f\colon \R^3\to \R^3$ be as in the previous section.

%Let now $ \phi \colon \R^3 \to \R^3$ be the monotone map associated to $f$, that is, $\phi  \circ R = f\circ \phi$, and let $(\psi_k)$ be the sequence of homeomorphisms, approximating $\phi $, which has the property that the restriction $\psi_k|_{|\sL_k|}$ is a twist map $\theta_w \colon |\sL_w| \to |\sL_w|$ for each word $w$ of length $k$ and that $\psi_k$ agrees with $\psi_{k-1}$ in the complement of $X_k$. Here and in what follows, we denote, for each $k\in \N$, $X_k = \bigcup_{|w|=k} |\sL_w|$.

By~\eqref{annanumero} construction, the involution $f\colon \R^3\to \R^3$ is locally $L(x)$-bilipschitz for $x\in \R^3\setminus \phi(X)$, where 
\[
L(x) = L_1 L_{\inner}^{2\eta(x)}/r_{\ell(x)}.
\]
Hence
\[
|Df(x)|\le L(x)
\]
for almost every $x\in \R^3\setminus \phi(X)$.

By \cite{Bing1988}, the number of corners of $\sT_w$ is at most double the corners of $\sL_w$ and the cubical loop $\sL_w$ has $4^{|w|}$ corners. Thus, for each $k\in \N$, we have also the estimate that
\[
|\{ x\in X_k \colon \eta(x)>0\}| \lesssim 4^{k+1} r_k^3 \le 4^{k+1} (10L_{\inner}^p)^{-k} \le (2L_{\inner}^p)^{-k} 
\]
for the Lebesgue measure of the set of those points of $X_k$ which are contained in the corners on some previous level. We conclude that
\[
\int_{\{ x\in X_k\setminus X_{k+1} \colon \eta(x)>0\}} |Df|^p \le (L_{\inner}^k)^p (2L_{\inner}^p)^{-k} = 2^{-k}.
\]
We estimate now the $L^p$-norm of $|Df|$ on the sets \[Y_k = \{ x\in X_k\setminus X_{k+1} \colon \eta(x)=0\} \, . \] Since $X_k$ is a pair-wise disjoint union of $2^k$ loops consisting at most $2/r_k$ cubes of side length $r_k^3$, we conclude that the Lebesgue measure $|X_k|$ of $X_k$ satisfies
\[
|X_k| \le 2^k (2/r_k)r_k^3 = 2^{k+1} r_k^2.
\]
Thus
\begin{eqnarray*}
%\int_{|\sL_0|\setminus \Cantor} |Df(x)|^p \dx &=& \sum_{k=0}^\infty 
\int_{Y_k} |Df|^p \dx \le (L_1/r_k)^p \vol(X_k) 
\le L_1 2^{k+1} r_k^{2-p} \le L_1 2^{k+1} 3^{-k} = 2L_1 (2/3)^k.
\end{eqnarray*}
We conclude that $|Df|$ is $L^p$-integrable on $X_0\setminus \phi(X)$. Since the Cantor set $\phi(X)$ has Lebesgue $3$-measure zero, we have that $|Df|$ is in $L^p$.

%By Lemma \ref{lemma:twist}, we have that each twist map $\theta_w$ is locally $L$-bilipschitz. Thus each homeomorphism $\psi_k$ is locally $L_j$-bilipschitz in $X_j \setminus X_{j+1}$, where 
%\[
%L_j = L_0 \max_{\substack{|w|=j\\ i=1,2}} \frac{\diam(|\sL_w|)}{\dist(|\sL_{wi}|, \partial |\sL_w|)} \le L_0 \frac{\diam(|\sL_0|)}{3^{-(m_j-1)}} = L_0 \diam(|\sL_0|) 3^{m_j-1}
%\]
%for $0\le j\le k$; here we used also \eqref{eq:distance}. We conclude that, by the cancellation lemma (Lemma \ref{lemma:cancellation}), 
%\[
%|Df(x)| \le L_k
%\]
%for almost every $x\in X_k\setminus X_{k-1}$.

%Since the Cantor set $\Cantor=q(\sB)$ has Lebesgue $3$-measure zero, we have
%\begin{eqnarray*}
%\int_{|\sL_0|\setminus \Cantor} |Df(x)|^p \dx &=& \sum_{k=0}^\infty \int_{X_k \setminus X_{k+1}} |Df|^p \dx \le \sum_{k=0}^\infty L_k^{p} \vol(X_k) \\
%&\le& C \sum_{k=0}^\infty \left( 3^{m_k-1}\right)^{p} \sum_{|w|=k} \left( \left( 3^{-m_k}\right)^3 \cdot \# \sL_w \right) \\
%&\le& C \sum_{k=0}^\infty 3^{p \cdot (m_k-1)} 2^k 3^{-2m_k} \\
%&\le& C \sum_{k=0}^\infty 2^k3^{(p-2)m_k} 
%= C \sum_{k=0}^\infty 2^k 3^{-k} <\infty.
%\end{eqnarray*}

It remains now to verify that $f$ is weakly differentiable. Let $f_k = \psi_k^{-1} \circ R \circ \psi_k$. Then $f_k$ is Lipschitz, by the previous calculation, $(f_k)$ is a Cauchy  sequence in $W^{1,p}(\R^3;\R^3)$. Furthermore, since  the sequence $(f_k)$ converges to $f$ locally uniformly, we conclude that $f\in W^{1,p}(\R^3,\R^3)$. This completes the proof.

\bibliographystyle{abbrv}
%\bibliography{Bing-Sobolev}

\end{document}